\documentclass[a4paper,11pt, reqno]{amsart}
\usepackage{amsmath}
\usepackage[utf8]{inputenc}
\usepackage[english]{babel}
\usepackage[babel]{csquotes}
\usepackage[backend=bibtex, firstinits=true, maxbibnames=99]{biblatex}
\usepackage[english]{varioref}  
\usepackage[bookmarks]{hyperref}
\usepackage{pdfpages}

\usepackage{amssymb, amsthm, amsfonts, braket, bm} 
\usepackage{mathrsfs, mathtools}
\usepackage{mathabx}

\usepackage{subfig, xcolor, booktabs, tabularx, graphicx, emptypage, tikz}
\usetikzlibrary{decorations.markings, arrows.meta}   

\usepackage{caption,  microtype, setspace}                     
\mathtoolsset{showonlyrefs}

\usepackage{articolo}

\bibliography{toto}
\AtEveryBibitem{
 \clearlist{address}
 \clearfield{date}
 \clearfield{isbn}
 \clearfield{issn}
 \clearlist{location}
 \clearfield{month}
 \clearfield{series}
 \clearfield{url}
 \clearfield{doi}

 \ifentrytype{book}{}{
  \clearlist{publisher}
  \clearname{editor}
 }
}

\newcommand{\GG}[1]{}

\providecommand{\PhaseMatrix}[1]{\begin{bmatrix} \cos({#1}) & -\sin({#1})\,\mhDelta^{-\frac12} \\ \sin({#1}) \mhDelta^\frac12 & \cos({#1}) \end{bmatrix}}

\newcommand{\Xplus}{{X^{\texttt{+}}}}
\newcommand{\Xminus}{{X^{\texttt{-}}}}
\newcommand{\xplus}{{x^{\texttt{+}}}}
\newcommand{\xminus}{{x^{\texttt{-}}}}

\newcommand{\Tcal}{{\mathcal{T}}}
\newcommand{\Mrom}{\boldsymbol{\mathrm{M}}}
\DeclareMathOperator{\dist}{d}

\newcommand{\PhaseSpace}{{\Hcaldot^{1/2}}}

\newcommand{\fboldstar}{\fbold_{\!\star}}
\newcommand{\fboldstarprime}{\fbold_{\!\star}'}

\newcommand{\fboldstartheta}{{{\Bd f}_{\theta}}}

\newcommand{\vstartheta}{v_{\theta}}
\newcommand{\vboldstartheta}{{{\Bd v}_{\theta}}}

\newcommand{\Fbold}{{\Bd F}}
\newcommand{\fboldbot}{{{\Bd f}_\bot}}
\newcommand{\fboldbotprime}{{\Bd{f}'_\bot}}
\newcommand{\mboldbot}{{{\Bd m}_\bot}}

\newcommand{\gboldbot}{{{\Bd g}_\bot}}
\newcommand{\hboldbot}{{{\Bd h}_\bot}}

\newcommand{\mbold}{{\Bd{m}}}
\newcommand{\wbold}{{\Bd{w}}}

\newcommand{\tautilde}{{\tilde{\tau}}}
\newcommand{\xitilde}{{\tilde{\xi}}}
\newcommand{\Wtilde}{{\tilde{W}}}

\DeclareMathOperator{\Span}{span}

\newcommand{\utildebold}{{\tilde{\ubold}}}

\newcommand{\obold}{{\Bd{0}}}

\newcommand{\smallpar}{\delta}
\newcommand{\Lfourthree}{{L^4(\R^{1+3})}}

\newcommand{\RPolar}{R}

\newcommand{\mhDelta}{\Tonde{-\Delta}}
\newcommand{\sqrtDelta}{\sqrt{-\Delta}}

\newcommand{\Lfour}{{L^4(\R^{1+3})}}
\newcommand{\Gammaalpha}{{\Gamma_{\!\alphabold}}}

\newcommand{\sigmabold}{\Bd{\sigma}}

\newcommand{\Arom}{{\mathrm{A}}}

\newcommand{\Pbold}{{\Bd P}}

\title[Asymptotics of the Strichartz norm for nonlinear waves]{A sharp Lorentz-invariant Strichartz norm expansion for the cubic wave equation in $\R^{1+3}$ }
\author{Giuseppe Negro}
\thanks{Supported by the MINECO grants SEV-2011-0087, SEV-2015-0554, SEV-2017-0718 and MTM2017-85934-C3-1-P, by the ERCEA Adv.\,Grant 2014 669689-HADE, by the BERC 2018-2021 program, and by the LAGA, Université~Paris 13.}
\address{BCAM - Basque Center for Applied Mathematics
}
\email{gnegro@bcamath.org}
\date{\today}
\begin{document}

\begin{abstract} 
We provide an asymptotic formula for the maximal Strichartz norm of small solutions to the cubic wave equation in Minkowski space. The leading coefficient is given by Foschi's sharp constant for the linear Strichartz estimate. We calculate the constant in the second term, which differs depending on whether the equation is focussing or defocussing. The sign of this coefficient also changes accordingly. 

\end{abstract}

\maketitle

\section{Introduction}
Considering solutions $v$ to the linear wave equation $\partial^2_{t}v=\Delta v$ in Minkowski space, 
Foschi \cite{Foschi07} found the best constant $\Scal_0=\frac{3}{16\pi}$ in the Strichartz inequality 
\begin{equation}\label{eq:StrichartzIneq}
	\norm{v}_{L^4(\R^{1+3})}^4 \le \Scal_0\norm{\vbold(t)}_{\Hcaldot^{1/2}(\R^3)}^4,
\end{equation} 
where $$\vbold(t)=\big(v(t), \partial_t v (t)\big)\quad  \text{and}\quad \Hcaldot^{1/2}(\R^3)=\Hdot^{1/2}\times \Hdot^{-1/2}(\R^3).$$ He also characterized the maximizers via symmetries of the inequality, including Lorentzian boosts.

Here we consider real-valued, global solutions $u$ to  the cubic equation
\begin{equation}\label{eq:cubic_NLW}\tag{$\mathrm{NLW}$}
	\begin{array}{cc}
   \partial^2_tu - \Delta u = \sigma u^3,\quad \text{on} \ \R^{1+3}, \\ 
   	\end{array}
\end{equation} 
where $\sigma\ne 0$. This equation is locally well-posed in $\PhaseSpace$, and small solutions are global; see the second section. It is well-known, and we present a proof in Appendix~B, that $\norm{\ubold(t)}_{\Hcaldot^{1/2}}$ is neither conserved in time, nor invariant under Lorentzian boosts. This has proved to be a fundamental obstruction; see~\cite{Do18, DL} .

In order to circumvent these difficulties, we consider
\begin{equation}\label{eq:I_functional}
  I(\delta) = \sup \Big \{\ \norm{u}_{L^4(\R^{1+3})}^4 \ \Big|\  \lim_{t\to-\infty} \norm{\ubold(t)}_{\Hcaldot^{1/2}(\R^3)}\le\delta \ \Big\},
\end{equation}
which is manifestly invariant under translations in time, and  we will prove in the third section that this is also invariant under Lorentzian boosts. 

Our main concern thereafter, will be the proof of  the following  sharp asymptotic estimate.
\begin{theorem}\label{thm:MainThree}
Let $\Scal_0=\frac{3}{16\pi}$ and let $\delta>0$ be sufficiently small.  Then the supremum in~\eqref{eq:I_functional} is attained and
    \begin{equation}\label{eq:MainThreeFormula} 
      I(\delta)= \Scal_0 \delta^4 +\sigma \Scal_1\delta^6 + O(\delta^8),
    \end{equation}
as $\delta \to 0$,  where
  \begin{equation}\label{eq:Scal_one}
    \Scal_1 = \begin{cases} \frac{29}{2^{10}\pi^3}, & \sigma>0 \quad (\text{focusing}),
    \\ \frac{5}{2^{10}\pi^3}, & \sigma<0 \quad (\text{defocusing}).\end{cases}
  \end{equation}
\end{theorem}
A similar asymptotic expansion was proven for the nonlinear Schr\"odinger equation, with $n=1$ or $2$, by Duyckaerts, Merle and Roudenko~\cite{DuMeRo11}; this is so far the only paper concerning maximizers for Strichartz norms in the nonlinear case. Our asymptotic analysis in the fourth section will be an adaptation of their argument. A key ingredient is the following version of the Strichartz inequality~\eqref{eq:StrichartzIneq}, in which the left-hand side is sharpened by adding a term proportional to the distance from the set~$\Mrom$ of maximizers. This was proved recently in~\cite{Negro18}. Consideration of inequalities of this type originated in a question of Brezis and Lieb~\cite[question~(c)]{BreLie85}, who asked whether the Sobolev inequality could be sharpened in the same way.
\begin{lem}\label{lem:SharpenedStrichartz}
Denote by $v=S\fbold$ the solution to $\partial_t^2 v=\Delta v$ with $\vbold(0)=\fbold$. Then there is a constant $c>0$ such that
  \begin{equation}\label{eq:SharpenedStrichartz}
    \norm{S\fbold}_{L^4(\R^{1+3})}^2 + c\dist(\fbold, \Mrom)^2\le \Scal_0^{1/2}\norm{\fbold}_{\Hcaldot^{1/2}(\R^3)}^2 ,
  \end{equation}
  where
  \begin{equation}\label{eq:MManifold}
  \Mrom:=\Set{ \gbold\in\Hcaldot^{1/2}(\R^3) | \norm{ S\gbold}_{L^4(\R^{1+3})}^4= \Scal_0 \norm{\gbold}_{\Hcaldot^{1/2}(\R^3)}^4 },
\end{equation}
and $\dist(\fbold, \Mrom):=\inf \big\{ \norm{\fbold-\gbold}_{\Hcaldot^{1/2}(\R^3)}\, |\, \gbold\in\Mrom\big\}$.
\end{lem}
In the fifth section, we use the Penrose transform to calculate the constant~$\Scal_1$. This step requires the explicit knowledge of the linear maximizers. In the sixth section, we will prove the existence of maximizers using a standard argument based on a nonlinear profile decomposition, which will be proved in Appendix~A. In the final section we give a partial result concerning the uniqueness of these maximizers. This requires the study of some geometrical properties of~$\Mrom$, which is carried out in Appendix~C.

There is intense research currently on the dynamics of the cubic wave equation~\eqref{eq:cubic_NLW} in $\Hcaldot^{1/2}$; see~\cite{Do16_1,Do16, DL, Ro17, Sh12} and the recent~\cite{Do18_1,Do18,DoLaMeMu18}.  However, to the knowledge of the author, the only paper, other than the present one, that deals with Lorentzian transformations is the work of Ramos~\cite{Ram18}; see also~\cite{KiStVi12} for the Klein-Gordon equation. 

The problem of finding sharp bounds for the Strichartz norm of solutions to nonlinear equations is open for large data. Duyckaerts and Merle~\cite{DuMe09b} obtained a sharp bound for solutions to the focusing quintic wave equation that are close to the \emph{threshold} solution. For the defocusing quintic wave equation in $\R^{1+3}$, Tao~\cite{Ta08} gives a bound of the $L^4(\R;L^{12}(\R^3))$ norm in terms of a tower of exponentials of the $\dot{H}^1\times L^2$ norms of initial data. This result holds for all data, not just small, but is unlikely to be sharp, and it is interesting to note that a much smaller bound had previously been given in the radial case by Ginibre, Soffer and Velo~\cite{GiSoVe92}.

\subsection*{Acknowledgements} This work formed part of my PhD thesis; I would like to express my gratitude to my directors, Thomas Duyckaerts and Keith Rogers. I would also like to thank Javier Ramos for helpful conversations, and the anonymous referee for helpful suggestions to improve the manuscript. 

\section{Preliminaries}

For a general function $w$ on $\R^{1+3}$, we will write $\wbold(t):=(w(t, \cdot), \partial_t w(t, \cdot))$. We use the box notation for the d'Alembert operator; 
\begin{equation}\label{eq:box}
	\Box\, w:= \partial^2_{t}w-\Delta w.
\end{equation}
For solutions to the linear equation $\Box\,v=0$ we will tend to use the letter $v$ and if the initial data is given $\vbold(0)=\fbold$, we denote $S\fbold=v$. 

We now turn to the definition of  a \emph{solution} to \eqref{eq:cubic_NLW}. Here we will consider only global solutions which scatter to linear solutions as $t\to -\infty$. The following operator is adapted to this.
\begin{defn}\label{def:antibox}
 For $F\in L^{4/3}(\R^{1+3})$, we define 
  \begin{equation}\label{eq:antibox}
    \antibox F(t, \cdot) = \int_{-\infty}^t \frac{\sin( (t-s)\sqrtDelta)}{\sqrtDelta}(F(s, \cdot))\, ds.
  \end{equation}
\end{defn}
\noindent This is well-defined because of the inhomogeneous Strichartz estimate, which follows by a standard duality argument from the Strichartz estimate of the introduction; see, for example,~\cite[Corollary~1.3]{KeTa98}.
\begin{prop}\label{prop:ConfStrichartz}
  Let $F\in L^{4/3}(\R^{1+3})$ and $w=\antibox F$. Then
  \begin{equation}\label{eq:ConfStrichartz}
    \norm{w}_{L^{4}(\R^{1+3})}+\sup_{t\in\R}\norm{\wbold(t)}_{{\Hcaldot^{1/2}}}\le 
    C \norm{F}_{L^{4/3}(\R^{1+3})}.
  \end{equation}
  Moreover, the map 
  \begin{equation}\label{eq:w_phasespace_cont}
    t\in\R\mapsto \wbold(t)\in \PhaseSpace(\R^3)
  \end{equation}
is continuous.
\end{prop}
\begin{rem}\label{rem:Inhomogeneous_Interval}
	Replacing $F$ with $F\mathbf1_{\{t<T\}}$, we immediately see that the following estimate also holds; 
	  \begin{equation}\label{eq:ConfStrichartzInterval}
	  	\begin{array}{cc}\Ds
    \norm{w}_{L^{4}((-\infty, T)\times \R^{3})}+\sup_{t\le T}\norm{\wbold(t)}_{{\Hcaldot^{1/2}}}\le 
    C \norm{F}_{L^{4/3}((-\infty, T)\times \R^{3})}, & \forall T\in \R.
    	\end{array}
  \end{equation}
\end{rem}
With this we obtain existence and uniqueness of small solutions by a standard application of the fixed-point theorem.
\begin{prop}\label{prop:small_data_theory}
	  There exists $\delta>0$ such that, if $\norm{\fbold}_{\PhaseSpace(\R^3)}\le \delta$, 
	  then there exists a unique solution $u$ to~\eqref{eq:cubic_NLW} that satisfies the condition 
	  \begin{equation}\label{eq:ancient_past} 
	    \lim_{t\to -\infty} \norm{ \ubold(t)-S\fbold(t)}_{\Hcaldot^{1/2}} = 0,
	  \end{equation}
 which we define as the fixed point of the mapping 
	\begin{equation}\label{eq:fixpoint}
		w\mapsto S \fbold + \sigma \antibox (w^3),
	\end{equation}
	in the space $L^4(\R^{1+3})\cap C(\R; \PhaseSpace(\R^3))$. 
	  Moreover, the nonlinear operator $$\Phi : \fbold \mapsto u $$ is locally bounded on $\PhaseSpace(\R^3)$, in the sense that
	  \begin{equation}\label{eq:crude_estimates}
	\norm{\Phi(\fbold)}_{\Lfourthree}+  \sup_{t\in\R} \norm{\Phi(\fbold)(t)}_{{\Hcaldot^{1/2}}} \le C_{\delta}\norm{\fbold}_{\Hcaldot^{1/2}}.
	  \end{equation}
	 \end{prop}
	In particular, we see that $I(\delta)$ is finite for small enough values of $\delta>0$. 
	 \begin{rem}\label{rem:Differentiability}
	 	The nonlinear operator $\Phi$ is also differentiable for $\norm{\fbold}_\PhaseSpace <\delta$. We denote its directional derivative by
		\begin{equation}\label{eq:DirDerivPhi}
			\begin{array}{cc}
	\Ds 			\Phi'(\fbold)\gbold:=\left.\frac{d}{d\eps}\Phi(\fbold+\eps \gbold)\right|_{\eps=0}, &\Ds \forall \gbold\in \PhaseSpace.
	\end{array}
	\end{equation}
	\end{rem}
	 \section{Lorentzian invariance}\label{sym}
For all $\alpha\in (-1, 1)$ we define a linear transformation of $\R^{1+3}$ as
\begin{equation}\label{eq:LorentzMatrix}
  L^\alpha(\tau, \xi_1, \xi_2, \xi_3)=\begin{bmatrix} \gamma & -\gamma \alpha & 0 & 0 \\ -\gamma \alpha & \gamma & 0 & 0 \\ 0 & 0& 1 & 0 \\ 0& 0 & 0 &1\end{bmatrix}\begin{bmatrix} \tau \\ \xi_1 \\ \xi_2 \\ \xi_3 \end{bmatrix},
\end{equation}
where $\gamma:=(1-\alpha^2)^{-{1/2}}$. Clearly, $\det L^\alpha=1$ and $(L^\alpha)^{-1}=L^{-\alpha}$; moreover, for all $(t, x), (\tau, \xi)\in \R^{1+3}$, 
\begin{equation}\label{eq:LorentzSymmetric}
  L^\alpha(\tau, \xi)\cdot(t, x)=(\tau, \xi)\cdot L^\alpha(t, x).
\end{equation}
Denoting $(\tautilde, \xitilde)=L^\alpha(\tau, \xi)$ we also have the fundamental property 
\begin{equation}\label{eq:LorentzNormPreserve}
	\tau^2-\abs{\xi}^2 = \tautilde^2 - |\xitilde|^2,
\end{equation}
from which it descends that, if $\tau=\abs{\xi}$, then $\tautilde=|\xitilde|$; to see this, note that $\tautilde^2=|\xitilde|^2$, and $\tautilde=\gamma\abs\xi -\gamma\alpha \xi_1\ge 0$. Analogously, if $\tau=-\abs{\xi}$ then $\tautilde=-|\xitilde|$. 

We also have the Dirac delta identity
 \begin{equation}\label{eq:DiracDeltaLorentz}
 	2\delta(\tau^2-\abs\xi^2)\mathbf 1_{\{\pm \tau >0\}}=\frac{\delta(\tau \mp \abs\xi)}{\abs \xi};
\end{equation}
see, for example,~\cite{Foschi07}. By the previous considerations, the left-hand side is Lorentz-invariant, and so 
\begin{equation}\label{eq:invariant_LHS}
	\frac{\delta(\tau \mp \abs\xi)}{\abs \xi}=\delta(\tau^2-\abs\xi^2)\mathbf 1_{\{\pm \tau >0\}}=\delta(\tautilde^2-|\xitilde|^2)\mathbf 1_{\{\pm \tautilde >0\}}=\frac{\delta(\tautilde \mp |\xitilde|)}{|\xitilde|},
\end{equation}
which implies the integration formula
  \begin{equation}\label{eq:LorentzVarchange}
   \int_{\R^3} F(L^\alpha(\pm\abs\xi, \xi) )G(\pm \abs\xi, \xi)\frac{d\xi}{|\xi|} = \int_{\R^3} F(\pm|\xitilde|, \xitilde)G(L^{-\alpha}(\pm|\xitilde|, \xitilde))\frac{d\xitilde}{|\xitilde|}.
  \end{equation}

We will now prove that $\antibox$ commutes with $L^\alpha$. It is for this reason that we defined $\antibox$ as an integral over  $(-\infty,t)$ rather that $(0,t)$. Ramos considered the operator as an integral over $(0,t)$, but in that case the operators do not commute precisely; see \cite[Proposition 1]{Ram18}. 
\begin{lem}\label{lem:LorentzCommute}
	Let $F\in L^{4/3}(\R^{1+3})$. Then, for all $\alpha\in(-1, 1)$,
  \begin{equation}\label{eq:antibox_commutes}
    \antibox(F\circ L^\alpha)=(\antibox F)\circ L^\alpha.
  \end{equation}
\end{lem}
\begin{proof}
By the definition \eqref{eq:antibox} and Fubini's theorem, $\antibox(F\circ L^\alpha)(t, x)$ can be written as
\begin{equation}\label{eq:Lorentz_antibox}
  \iiint\frac{\sin((t-s)\abs\xi)}{\abs{\xi}}e^{i(x-y)\cdot\xi} F(L^\alpha(s,y))\mathbf1_{\{ s <t\}}\, dsdy\frac{d\xi}{\abs \xi},
\end{equation}
modulo irrelevant factors of $(2\pi)^{-3}$. On the other hand,  we divide the operator
\begin{equation}\label{eq:antibox_decomp}
  \antibox = \antibox_+  - \antibox_-,  
\end{equation}
where, for an arbitrary $H\in L^{4/3}(\R^{1+3})$,
\begin{equation}\label{eq:pm_antibox}
  \antibox_{\pm} H (t, x) := \iiint \frac{e^{i(t, x)\cdot(\pm\abs \xi, \xi) - i(s, y)\cdot(\pm\abs \xi, \xi)}}{2i} H(s, y) \mathbf1_{\{s<t\}}\, dsdy\frac{d\xi}{\abs \xi}.
\end{equation}
We compute a convenient expression for $(\antibox_{\pm} F)(L^\alpha(t, x))$ using the properties of $L^\alpha$ that we recalled in the beginning of the section;
\begin{equation}\label{eq:pm_antibox_Lorentz}
  \begin{split}
   &\ \iiint \frac{e^{iL^\alpha(t, x)\cdot(\pm\abs \xi, \xi) - i(s, y)\cdot(\pm\abs \xi, \xi)}}{2i} F(s, y) \mathbf1_{\{s<\gamma t-\gamma \alpha x_1\}}\, dsdy\frac{d\xi}{\abs \xi} \\ 
  = &\ \iiint \frac{e^{i(t, x)\cdot L^\alpha(\pm\abs \xi, \xi) - i(s, y)\cdot(\pm\abs \xi, \xi)}}{2i} F(s, y) \mathbf1_{\{s<\gamma t-\gamma \alpha x_1\}}\, dsdy\frac{d\xi}{\abs \xi} \\
    =&\ \iiint \frac{e^{i(t, x)\cdot (\pm\abs \xi, \xi) - i(s, y)\cdot L^{-\alpha}(\pm\abs \xi, \xi)}}{2i} F(s, y) \mathbf1_{\{s<\gamma t-\gamma \alpha x_1\}}\, dsdy\frac{d\xi}{\abs \xi} \\
     =&\ \iiint \frac{e^{i(t, x)\cdot (\pm\abs \xi, \xi) - i L^{-\alpha}(s, y)\cdot(\pm\abs \xi, \xi)}}{2i} F(s, y) \mathbf1_{\{s<\gamma t-\gamma \alpha x_1\}}\, dsdy\frac{d\xi}{\abs \xi} \\
  =&\ \iiint \frac{e^{i(t, x)\cdot (\pm\abs \xi, \xi) - i(s, y)\cdot(\pm\abs \xi, \xi)}}{2i} F(L^\alpha(s, y)) \mathbf1_{\{\gamma s-\gamma \alpha y_1<\gamma t-\gamma \alpha x_1\}}\, dsdy\frac{d\xi}{\abs \xi}.
   \end{split}  
\end{equation}
\begin{figure}[!]
\centering
\begin{tikzpicture}
\fill[green!20!white] (-2, 0) -- (0,0) -- (-2,-1) -- cycle;
\fill[green!20!white] (2, 0) -- (0,0) -- (2,1) -- cycle;
\fill[gray!20!white] (-2, -2) -- (2, 2) -- (-2, 2) -- (2, -2) -- cycle;
\draw[->] (-2, 0) -> (2.2,0) node[scale=0.8, anchor=west] {$y_1$};
\draw[->] (0, -2) -> (0, 2.2) node[scale=0.8, anchor=south]{$s$};
\draw (-2, -2) -- (2, 2) node[scale=0.8, anchor=south west] {$\abs{s}=\lvert y_1\rvert$}; 
\draw (-2, 2) -- (2, -2);
\draw (-2, -1) -- (2, 1) node[scale=0.8, anchor=west] {$s=\alpha y_1$};
\end{tikzpicture}
\caption{The support of $\mathbf1_{\{s<\alpha y_1\}}-\mathbf1_{\{s<0\}}$ (green) intersects the light cone (gray) only at the origin.}
\label{fig:LorentzSupport}
\end{figure}
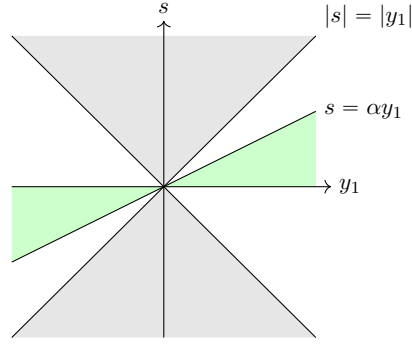
We conclude that $(\antibox F)(L^\alpha(t, x))$ is equal to  
\begin{equation}\label{eq:antibox_Lorentz}
  \iiint\frac{\sin((t-s)\abs\xi)}{\abs{\xi}}e^{i(x-y)\cdot\xi} F(L^\alpha(s,y))\mathbf1_{\{ s < t- \alpha (x_1-y_1)\}}\, dsdy\frac{d\xi}{\abs \xi}.
\end{equation}

Using these two expressions, the difference $\antibox(F\circ L^\alpha)-(\antibox F)\circ L^\alpha$ can be written as
\begin{equation}\label{eq:LorentzCommutator}
 \iiint \frac{\sin(s\abs\xi)}{\abs \xi} e^{-iy\cdot \xi} G(s, y) \Tonde{ \mathbf1_{\{s<\alpha y_1\}} - \mathbf1_{\{s<0\}}}\,dsdyd\xi,
\end{equation}
where  $G(s, y):=F(L^\alpha(s+t, y+x))$.
We now note that the distribution $v$, defined by the formal integral 
\begin{equation}\label{eq:Fourier_fund_sol}
  v(s, y):=\int_{\R^3}\frac{\sin(s\abs\xi)}{\abs\xi}e^{-iy\cdot \xi}\, d\xi,
\end{equation}
is a fundamental solution to the wave equation, that is,
\begin{equation}\label{eq:fund_wave_eqn}
  \begin{cases} \Box\, v=0, & \text{on }\R^{1+3},\\ 
   \vbold(0)=(0, \delta),
  \end{cases}
\end{equation}
where $\delta$ is the Dirac distribution. Therefore, $v$ is supported in the cone $\{|y|^2\le s^2\}$, which intersects the support of $\mathbf1_{\{s<\alpha y_1\}} - \mathbf1_{\{s<0\}}$ only at the origin (recalling that $|\alpha|<1$); see Figure~\ref{fig:LorentzSupport}. 
Thus the integral~\eqref{eq:LorentzCommutator} vanishes, completing the proof.
\end{proof}                                                                                                                                                                                                                                                                                                                
\begin{cor}\label{cor:LorentzPhaseSpaceContinuous}
 Let $\alpha\in(-1, 1)$, let $F\in L^{4/3}(\R^{1+3})$, and let $w_\alpha=\antibox F\circ L^\alpha$. Then the map $t\in \R\mapsto \wbold_\alpha(t)\in \PhaseSpace(\R^3)$ is continuous.
\end{cor}

The full symmetry group of solutions to~\eqref{eq:cubic_NLW} that we consider in this paper consists of Lorentzian boosts, dilations and spacetime translations. The Lorentzian boost of velocity $\beta\in \R^{3}$, with $|\beta|<1$, is defined by
\begin{equation}\label{eq:LorentzBoost}
	\begin{array}{cc}
  L^\beta(\tau, \xi)=R^{-1}\circ L^\alpha\circ  R (\tau, \xi),& \text{where } \alpha=\abs\beta,
  	\end{array} 
\end{equation}
and $R(\tau, \xi)=(\tau, R'\xi)$, with $R'$ being a rotation that maps $(1,0,0)$ to $\beta/|\beta|$. By convention we assume that $L^{(0,0,0)}$ is the identity. We denote
	\begin{equation}\label{eq:Lambda_Poinc}
	  \Lambda(t, x)= L^\beta\big(\lambda (t-t_0), \lambda (x-x_0)\big), 
	\end{equation}
where $t_0\in \R, x_0\in \R^3, \lambda>0$ and $\beta\in\R^3$, with $\abs{\beta}<1$; note that Lemma~\ref{lem:LorentzCommute} readily implies that, for all $F\in L^{4/3}(\R^{1+3})$, 
\begin{equation}\label{eq:GeneralAntiboxCommute}
	\antibox(F\circ \Lambda)=\lambda^{-2}(\antibox F)\circ \Lambda.
\end{equation}
It is well-known that these transformations act unitarily on solutions to the linear wave equation with data in $\PhaseSpace$, as in the following lemma. For a proof, see, for example, the third section of~\cite{Negro18}.
\begin{lem}\label{lem:LinearPoincareAction}
	Let $\fbold\in\PhaseSpace(\R^3)$. There exists a unique $\fbold_\Lambda\in\PhaseSpace(\R^3)$ such that 
	\begin{equation}\label{eq:LinearPoincareAction}
		\lambda S\fbold(\Lambda(t, x)) = S\fbold_\Lambda(t,x).
	\end{equation}
	Moreover,  $\norm{\fbold}_{{\Hcaldot^{1/2}}}=\norm{\fbold_\Lambda}_{{\Hcaldot^{1/2}}}$.
\end{lem}
The transformation $\Lambda$ also maps smooth solutions of~\eqref{eq:cubic_NLW} to smooth solutions. Using Lemma~\ref{lem:LorentzCommute}, we can now describe the action of $\Lambda$ on the class of solutions that we defined in Proposition~\ref{prop:small_data_theory}.
\begin{thm}\label{thm:NonlinearPoincareAction}
 Let $u\in L^4(\R^{1+3})$, with $\ubold\in C(\R;\PhaseSpace)$, satisfy the fixed point equation $u=S\fbold+\sigma\antibox(u^3)$. Denote 
\begin{equation}\label{eq:NonlDilPoinc}
	u_\Lambda(t, x)=\lambda u(\Lambda(t, x)).
\end{equation} 
Then $u_\Lambda \in L^4(\R^{1+3})$, with $\norm{u_\Lambda}_{L^4}=\norm{u}_{L^4}$, $\ubold_\Lambda\in C(\R; \PhaseSpace)$ and
\begin{equation}\label{eq:NonlinPoincTransf}
	u_\Lambda = S\fbold_\Lambda + \sigma \antibox( u_\Lambda^3),
\end{equation}
where $\fbold_\Lambda$ is defined in~\eqref{eq:LinearPoincareAction}; in particular,
 \begin{equation}\label{eq:nonlinear_norm_preserving}
 	\lim_{t\to-\infty} \norm{\ubold_\Lambda(t)}_{{\Hcaldot^{1/2}}} =\lim_{t\to -\infty} \norm{\ubold(t)}_{\Hcaldot^{1/2}}. \end{equation}
\end{thm}
\begin{proof}
Using~\eqref{eq:GeneralAntiboxCommute}, we obtain from $u=S\fbold + \sigma \antibox(u^3)$ that 
\begin{equation}\label{eq:nonlin_change_var}
  \begin{split}
    \lambda u \circ \Lambda &= \lambda (S\fbold)\circ \Lambda + \lambda \sigma \antibox (u^3)\circ \Lambda \\ 
    &=S\fbold_\Lambda +\sigma\antibox(u_\Lambda^3), 
  \end{split}
\end{equation}
which proves~\eqref{eq:NonlinPoincTransf}. The fact that $\ubold_\Lambda\in C(\R;\PhaseSpace)$ follows from Corollary~\ref{cor:LorentzPhaseSpaceContinuous}.
\end{proof}

\section{The asymptotic formula}\label{asym}

Throughout this section, we consider $\norm{\fbold}_{\Hcaldot^{1/2}}\le \delta$ with $\delta$  sufficiently small, so that the corresponding solution $u=\Phi(\fbold)$   is well-defined, by Proposition~\ref{prop:small_data_theory}. Recalling that 
\begin{equation}\label{eq:u_delta_recall}
	u=\Phi(\fbold)=S\fbold + \sigma \antibox(u^3),
\end{equation}
we will require the following estimates on Picard iterations.
	 \begin{lem}\label{lem:linear_nonlinear_comparison}
	Let  $\norm{\fbold}_{\Hcaldot^{1/2}}\le\delta$. Then  as $\delta \to 0$,  
	  \begin{align}\label{eq:CompareFirstOrder}
	    \Phi(\fbold)&=S\fbold+O(\smallpar^{3}), \\
	\label{eq:CompareSecondOrder}
	   \Phi(\fbold)&=S \fbold+\sigma\antibox\Tonde{(S\fbold)^3 }+O(\smallpar^5),
	  \end{align}
	  where the big-$O$ symbols refer to the norms of $L^4(\R^{1+3})$ and $C(\R; \PhaseSpace)$.
	 \end{lem}
\begin{proof}
By the final estimate of Proposition~\ref{prop:small_data_theory}, we have $u =\Phi(\fbold)=O(\delta)$ and so $\norm{u^3}_{L^{4/3}}=O(\delta^3)$. Then, by the Strichartz estimate of Proposition~\ref{prop:ConfStrichartz}, we obtain $$\antibox(u^3)=O(\delta^3),$$ so the fixed point equation~\eqref{eq:u_delta_recall} yields~\eqref{eq:CompareFirstOrder}. Now, by the Hölder inequality,
\begin{equation}\label{eq:estimate_ucube_Sfcube}
	\norm{u^3-(S\fbold)^3}_{L^{4/3}}\le C\norm{u-S\fbold}_{L^4}\Tonde{\norm{u}_{L^4}^2+\norm{S\fbold}_{L^4}^2}\le O(\delta^5),
\end{equation}
where we used~\eqref{eq:CompareFirstOrder} to estimate $u-S\fbold$. We rewrite this as $$u^3=(S\fbold)^3 + O(\delta^5),$$ where the big-O symbol refers to the $L^{4/3}$ norm, and inserting this into the fixed point equation yields~\eqref{eq:CompareSecondOrder}. 
\end{proof}                                                                                                                                                                                                                                                                                                                

The function $I$, defined in the introduction,  can be rewritten as
\begin{equation}\label{eq:I_functional_fourth_sect}
  I(\delta) = \sup \Big \{\ \norm{\Phi(\fbold)}_{L^4(\R^{1+3})}^4 \ \Big|\  \norm{\fbold}_{\Hcaldot^{1/2}(\R^3)}\le\delta \ \Big\}.
\end{equation}
We record some properties of the $\fbold$ that come close to maximize $I(\delta)$; in the proof, we will need the sharpened Strichartz estimate of Lemma~\ref{lem:SharpenedStrichartz}.
	\begin{lem}\label{lem:FirstOrder}
		Let $\norm{\fbold}_{\Hcaldot^{1/2}}\le \delta$ and $u=\Phi(\fbold)$ be close to maximal in the sense that
		\begin{equation}\label{eq:almost_maximizer}
			I(\delta)- \|u\|^4_{L^4(\R^{1+3})}=O(\delta^6).
		\end{equation}
		Then $\norm{\fbold}_\PhaseSpace=\delta + O(\delta^3)$ and $\dist(\fbold, \Mrom)=O(\delta^2)$. Moreover, there is a $C>0$ such that
		\begin{equation}\label{eq:CoFi}
			\norm{ S\fbold}_{L^4(\R^{1+3})}^4\le \Scal_0 \delta^4 -C \delta^2\dist(\fbold, \Mrom)^2.
		\end{equation}
	\end{lem}
\begin{proof}
	By squaring the sharpened Strichartz estimate \eqref{eq:SharpenedStrichartz}, we obtain
	\begin{equation}\label{eq:SharpenedStrichartzUsable}
	  \norm{S\fbold}_{L^4}^4 + 2c \norm{S\fbold}_{L^4}^2\dist(\fbold, \Mrom)^2\le \Scal_0\delta^4.
	\end{equation}
Now, we use the first Picard estimate \eqref{eq:CompareFirstOrder} for $u=\Phi(\fbold)$ in order to find upper and lower bounds for $I(\delta)$. On the one hand, by combining it with the closeness assumption~\eqref{eq:almost_maximizer} and with~\eqref{eq:SharpenedStrichartzUsable}, we find that
	\begin{equation}\label{eq:UpBoundFirstOrder}
	  \begin{split}
			I(\delta)=\norm{u}_{L^4}^4 +O(\delta^6)
			 &= \norm{ S\fbold }_{L^4}^4 + O(\delta^6)  \\
			&\le \Scal_0 \delta^4  -2c\norm{S\fbold}_{L^4}^2 \dist(\fbold, \Mrom)^2 +O(\delta^6).
	\end{split}
	\end{equation}
	 On the other hand, if $\gbold\in\Mrom$ is such that $\norm{\gbold}_{\Hcaldot^{1/2}(\R^3)}=1$, then, by definition,
	 \begin{equation}\label{eq:LowBoundFirstOrder}
	 	I(\delta)\ge \norm{ \Phi(\delta \gbold)}_{L^4}^4 \ge\Scal_0 \delta^4 +O(\delta^6),
	\end{equation}
	where the second inequality uses \eqref{eq:CompareFirstOrder} and the fact that $\norm{S(\delta \gbold)}_{L^4}^4=\Scal_0\delta^4$.  Combining these upper and lower bounds for $I(\delta)$ we find that
	\begin{equation}\label{eq:dist_bound}
		2c\norm{S \fbold}_{L^4}^2\dist(\fbold,\Mrom)^2\le O(\delta^6),
	\end{equation}
and
	\begin{equation}\label{eq:Sf_is_deltafour}
	  \norm{S\fbold}_{L^4}^4\ge \Scal_0\delta^4+ O(\delta^6).
	\end{equation}
Using the Strichartz inequality $\Scal_0\norm{\fbold}_\PhaseSpace^4\ge \norm{S\fbold}_{L^4}^4$ and the assumption $\norm{\fbold}_\PhaseSpace\le \delta$, the bound~\eqref{eq:Sf_is_deltafour} gives that $\norm{\fbold}_\PhaseSpace=\delta + O(\delta^3)$. Inserting~\eqref{eq:Sf_is_deltafour} into~\eqref{eq:dist_bound} we conclude that $\dist(\fbold, \Mrom)^2=O(\delta^4)$. On the other hand, reinserting~\eqref{eq:Sf_is_deltafour} into~\eqref{eq:SharpenedStrichartzUsable} yields~\eqref{eq:CoFi}, and the proof is complete.
\end{proof}

For a slightly stronger version of the following lemma, see Proposition~\ref{prop:Metric_Projection} in Appendix~\ref{app:geometry}.

\begin{lem}\label{lem:metric_proj}
  For every $\fbold\in\PhaseSpace(\R^3)$ there exists a $\fbold_{\!\star}\in\Mrom$ such that
  \begin{equation}\label{eq:fboldbot_distance}
    \norm{\fbold-\fbold_{\!\star}}_{{\Hcaldot^{1/2}}(\R^3)}=\dist(\fbold, \Mrom).
  \end{equation}
  Moreover, $\Braket{\fbold_{\!\star}|\fbold-\fbold_{\!\star}}_{\Hcaldot^{1/2}}=0$ and we write $\fboldbot:=\fbold-\fbold_{\!\star}$; see Figure~\ref{fig:distproj}.
\end{lem}
\begin{figure}
\centering 
\scalebox{0.7}
{
\begin{tikzpicture}
[
tangent/.style={
        decoration={
            markings,
            mark=
                at position #1
                with
                {
                    \coordinate (tangent point-\pgfkeysvalueof{/pgf/decoration/mark info/sequence number}) at (0pt,0pt);
                    \coordinate (tangent unit vector-\pgfkeysvalueof{/pgf/decoration/mark info/sequence number}) at (1,0pt);
                    \coordinate (tangent orthogonal unit vector-\pgfkeysvalueof{/pgf/decoration/mark info/sequence number}) at (0pt,1);
                }
        },
        postaction=decorate
    },
    use tangent/.style={
        shift=(tangent point-#1),
        x=(tangent unit vector-#1),
        y=(tangent orthogonal unit vector-#1)
    },
    use tangent/.default=1
]
	\draw (0, 0) -- (3, 5)         ;
	\draw (0,0) -- (-3, 5)         ;
	\draw (0, 5) circle (3 and 1/3) ;
	\draw (0, 5) node {\Mrom}; 
	\fill ({atan(5/3)}: {2/3*sqrt(34)}) node (fboldstar) {} node [xshift=-8, yshift=7]  {$\fbold_{\!\star}$};

	\draw (0,0) node[anchor=west] {$\obold$};
	 \draw [opacity=0, tangent=0]( {atan(5/3)}: {2/3*sqrt(34)}) arc (0:-135:2/3*3 and 2/3*1/3) ; 
	\draw [use tangent,rotate=-30, -{>[scale=1.5]}] (0,0) --node[sloped, rotate=-35, above] {\color{black} $\fboldbot$} (0,2) node (GammaMinus) {}  node [anchor=north west] {$\fbold$};
	\fill 
	  (fboldstar) circle (0.05) (GammaMinus) circle (0.05) (0,0) circle (0.05);
	\draw[use tangent, rotate=-30, help lines] 
		(0,2) -- + (1,0) 
		 (1,0) 
		++(-0.1, 0) --node[sloped, rotate=-30, below] {\color{black} $\dist(\fbold, \Mrom)$} +(0, 2); 
\end{tikzpicture}
}
\caption{ }
\label{fig:distproj}
\end{figure}

We can now obtain the asymptotic formula by combining the previous lemmas with the second Picard iteration estimate.
\begin{prop}\label{prop:SecondOrder}
Let $\norm{\fbold}_{\Hcaldot^{1/2}}\le\delta$ and $u=\Phi(\fbold)$ be close to maximal in the sense that
  \begin{equation}\label{eq:almost_maximizer_stronger}
	I(\delta)- \|u\|^4_{L^4(\R^{1+3})}=O(\delta^8).
 \end{equation}
 Then $\dist(\fbold, \Mrom) = O(\delta^3)$ and, as $\delta\to0$,  
  \begin{equation}\label{eq:Taylor_second_order}
    I(\delta) = \Scal_0 \delta^4 + \sigma \Scal_1\delta^6 +O(\delta^8), 
  \end{equation}
  where $\sigma$ is the coefficient of the nonlinearity in~\eqref{eq:cubic_NLW}. The constant $\Scal_1$ satisfies 
  \begin{equation}\label{eq:ScalOne}
    \sigma\Scal_1= \sup \Set{ \sigma \iint_{\R^{1+3}} (S\gbold)^3\antibox((S\gbold)^3)\, dtdx  |  
      \begin{array}{c} \gbold\in \Mrom \\ \norm{\gbold}_{{\Hcaldot^{1/2}}}=1\end{array} }.
  \end{equation}
\end{prop}

\begin{proof}  
By Lemma~\ref{lem:metric_proj}, we can write $\fbold=\fbold_{\!\star}+\fboldbot$. Using the orthogonality, we have
$$\norm{\fbold_{\!\star}}_{\Hcaldot^{1/2}}^2+\norm{\fboldbot}_{\Hcaldot^{1/2}}^2=\|\fbold\|_{\Hcaldot^{1/2}}^2\le \delta^2,$$ from which we conclude that
$  \norm{\fbold_{\!\star}}_{\Hcaldot^{1/2}}\le \delta$. This also shows that 
\begin{equation}\label{eq:fboldstar_norm_estimate}
	\norm{\fboldstar}_\PhaseSpace^2 = \delta^2 + O(\delta^4),
\end{equation} 
because $\norm{\fbold}^2_\PhaseSpace=\delta^2 +O(\delta^4)$ and $\norm{\fboldbot}_\PhaseSpace^2 = O(\delta^4)$ by Lemma~\ref{lem:FirstOrder}. Expanding, we find
\begin{equation}\label{eq:Sf_delta_cube}
  \Tonde{S\fbold}^3= \Tonde{S \fbold_{\!\star}}^3 + O(\delta^2\norm{\fboldbot}_{\Hcaldot^{1/2}}),
\end{equation}
where the big-O symbol refers to the $L^{4/3}(\R^{1+3})$ norm. Applying $\antibox$, we infer from the Strichartz estimates~\eqref{eq:ConfStrichartz} that
\begin{equation}\label{eq:antibox_Sf_delta_cube}
 \antibox(\Tonde{S\fbold}^3)= \antibox(\Tonde{S \fbold_{\!\star}}^3) + O(\delta^2\norm{\fboldbot}_{\Hcaldot^{1/2}}),
\end{equation}
where the big-O now refers to both the $L^4(\R^{1+3})$ and the $C(\R;\PhaseSpace)$ norm. So, we can write 
\begin{equation}\label{eq:DecompAftermath}
\iint_{\R^{1+3}}(S\fbold)^3\antibox((S\fbold)^3)=\iint_{\R^{1+3}}(S\fboldstar)^3\antibox((S\fboldstar)^3) + O(\delta^5\norm{\fboldbot}_{{\Hcaldot^{1/2}}}).
\end{equation}

Now the key ingredient in this case is the second Picard estimate~\eqref{eq:CompareSecondOrder}, from which we deduce $$\norm{\Phi( \hbold)}_{L^4}^4=\norm{ S \hbold + \sigma\antibox((S \hbold)^3)}_{L^4}^4+ O(\delta^8),$$
whenever $\norm{\hbold}_{\Hcaldot^{1/2}}\le \delta$. This implies that 
 \begin{equation}\label{eq:ExpSecOrder}
    \norm{\Phi( \hbold)}_{L^4}^4= \norm{S \hbold}_{L^4}^4+ 4\sigma\iint_{\R^{1+3}}(S\hbold)^3\antibox((S\hbold)^3) + O(\delta^8).
 \end{equation}
As $u=\Phi( \fbold)$ with $\norm{\fbold}_{\Hcaldot^{1/2}}\le\delta$, on the one hand this yields an upper bound using our closeness hypothesis;
 \begin{equation}\label{eq:FirstPicard}
   I(\delta)\le \norm{u}_{L^4}^4 + O(\delta^8) = \norm{S \fbold}_{L^4}^4 +4\sigma\iint_{\R^{1+3}}(S\fbold)^3\antibox((S\fbold)^3) + O(\delta^8).
 \end{equation}
Estimating the first term on the right-hand side using~\eqref{eq:CoFi} of the previous lemma and the second term using~\eqref{eq:DecompAftermath}, we obtain
\begin{equation}\label{eq:PollutedUpBound}
 \begin{split}
 I(\delta)\le\  & \Scal_0\delta^4 +4\sigma\iint_{\R^{1+3}}(S\fboldstar)^3\antibox((S\fboldstar)^3)  - C \delta^2 \dist(\fbold, \Mrom)^2\\
   & +O(\delta^5\dist(\fbold, \Mrom))+O(\delta^8).
  \end{split}
\end{equation}
For the lower bound, we let $\tilde{\fboldstar}:= \fboldstar/\norm{\fboldstar}_{\Hcaldot^{1/2}}$, so that $I(\delta)\ge \lVert \Phi(\delta\tilde\fboldstar)\rVert_{L^4}^4$, and expanding using~\eqref{eq:ExpSecOrder} we obtain 
\begin{equation}\label{eq:IdeltaBelow}
  I(\delta)\ge \Scal_0\delta^4 +4\sigma\delta^6\iint_{\R^{1+3}}(S\tilde\fboldstar)^3\antibox((S\tilde\fboldstar)^3)+ O(\delta^8),
\end{equation}
where we used that $\lVert S\tilde{\fboldstar}\rVert_{L^4}^4=\Scal_0$. Now, using~\eqref{eq:fboldstar_norm_estimate}, we see that 
\begin{equation}\label{eq:Scal_estimate}
	\begin{split}
	\delta^6\iint_{\R^{1+3}}(S\tilde\fboldstar)^3\antibox((S\tilde\fboldstar)^3)&=\norm{\fboldstar}^6_\PhaseSpace\iint_{\R^{1+3}}(S\tilde\fboldstar)^3\antibox((S\tilde\fboldstar)^3) + O(\delta^8)\\
	& = \iint_{\R^{1+3}}(S\fboldstar)^3\antibox((S\fboldstar)^3)+O(\delta^8),
	\end{split}
\end{equation}
so combining the upper and lower bounds~\eqref{eq:PollutedUpBound} and~\eqref{eq:IdeltaBelow} yields
\begin{equation}\label{eq:DistBootstrap}
  \delta^2\dist(\fbold, \Mrom)^2\le O(\delta^5\dist(\fbold, \Mrom) + \delta^8).
\end{equation}
Writing $X:=\dist(\fbold, \Mrom)\delta^{-3}$, this  reads $X^2\le O(1+X)$, which implies that $X=O(1)$. Thus we find that $\dist(\fbold, \Mrom)=O(\delta^3)$.

To complete the proof we observe that, since $O(\delta^5\dist(\fbold, \Mrom))\!=\!O(\delta^8)$, it follows from~\eqref{eq:PollutedUpBound} and~\eqref{eq:IdeltaBelow} that 
\begin{equation}\label{eq:IntUdeltaTaylor}
 I(\delta) = \Scal_0 \delta^4 +4\sigma  \iint_{\R^{1+3}}(S\fboldstar)^3\antibox((S\fboldstar)^3) + O(\delta^8).
\end{equation}
However, for all $\gbold\in\Mrom$ with $\norm{\gbold}_{\Hcaldot^{1/2}}=\delta$, we also have
\begin{equation}\label{eq:Idelta_Dominates}
	I(\delta)\ge \norm{\Phi(\delta\gbold)}_{L^4}^4 = \Scal_0\delta^4 +4\sigma\iint_{\R^{1+3}}(S\gbold)^3\antibox((S\gbold)^3)+O(\delta^8),
\end{equation}
and so, combining this with \eqref{eq:IntUdeltaTaylor}, we conclude that the term
\begin{equation}\label{eq:SecOrdTerm}
	\sigma\iint_{\R^{1+3}}(S\fboldstar)^3\antibox((S\fboldstar)^3)
\end{equation}
must be equal to
 \begin{equation}\label{eq:SecOrdInt}
 	\begin{split}
   \sup \Set{\sigma \iint_{\R^{1+3}}(S\gbold)^3\antibox((S\gbold)^3) | \begin{array}{c}\gbold\in \Mrom\\ \norm{\gbold}_{\Hcaldot^{1/2}}=\delta\end{array}}+ O(\delta^8),
    \end{split}
 \end{equation}
thus proving~\eqref{eq:ScalOne}.
\end{proof}                                                                                                                                                                                                                                                                                                                

It remains to evaluate this supremum, which we will do in the sequel.

\section{Computation of the constant $\Scal_1$ via the Penrose transform }\label{pen}
We consider the following family of elements of $\PhaseSpace(\R^3)$:
  \begin{equation}\label{eq:extremizer_seed_conformal_notation}
	\fboldstartheta:=\Tonde{\cos \theta \frac{2}{1+\abs{\cdot}^2}, \sin\theta \Tonde{\frac{2}{1+\abs{\cdot}^2}}^2},
\end{equation}
and we let
\begin{equation}\label{eq:ustartheta_fboldtheta}
  \begin{array}{cc}
 \vstartheta:=S \fboldstartheta, & \vboldstartheta:=(\vstartheta, \partial_t \vstartheta).
  \end{array}
\end{equation}
One can calculate that $\norm{\fboldstartheta}_{\Hcaldot^{1/2}}=\abs{\SSS^3}^{1/2}$; see~\cite[equation~(33)]{Negro18}. \begin{rem}\label{rem:Ph_no_symmetry}
For all $t\in\R$ it holds that $\vboldstartheta(t)=\Ph_\theta\vbold_{0}(t)$, where 
  \begin{equation}\label{eq:PhaseSymmetry}
	\Ph_\theta \fbold: = \PhaseMatrix{\theta} \begin{bmatrix} f_0 \\ f_1 \end{bmatrix}.
\end{equation}
The operator $\Ph_\theta\colon \PhaseSpace \to \PhaseSpace$ is unitary and it commutes with the linear propagator $S$, but it does \emph{not} commute with the nonlinear propagator $\Phi$. 
\end{rem}
\begin{prop}[Foschi~\cite{Foschi07}]\label{prop:StriMaxCharacterize}
Let $\Mrom$ be the set of extremizing functions for the Strichartz inequality; see~\eqref{eq:MManifold}. Then
\begin{equation}\label{eq:MromCharacterized}
  \Mrom=\Set {c\left.(\vboldstartheta \circ \Lambda)\right|_{t=0} | c, \theta, \Lambda} ,
\end{equation}
where $c\ge 0$, $\theta\in \SSS^1$ and $\Lambda(t, x)=L^\beta\big(\lambda (t-t_0), \lambda (x-x_0)\big)$.
\end{prop}
\begin{rem}\label{rem:uniquely_determined}
	If $c\left.(\vboldstartheta \circ \Lambda)\right|_{t=0} =c' \left.(\vbold_{\theta'} \circ \Lambda')\right|_{t=0}$ and $c\ne 0$, then $c=c', \theta=\theta'$ and $\Lambda=\Lambda'$; see Appendix~\ref{app:geometry}.
\end{rem}
Recalling the definition \eqref{eq:ScalOne} of $\Scal_1$, we define 
\begin{equation}\label{eq:ScalFunctional}
	\begin{array}{cc}
 \Ds \Scal(w):=\iint_{\R^{1+3}}w^3 \antibox(w^3), & \text{where } w\in L^4(\R^{1+3}),
 	\end{array} 
\end{equation}
so that $\sigma \Scal_1=\sup \{\sigma \Scal(v)\ |\ v=S \gbold,\ \gbold\in \Mrom,\ \norm{\gbold}_{{\Hcaldot^{1/2}}}=1\}$.
\begin{prop}\label{prop:SoneIsSonetheta}
 For all $w\in L^4(\R^{1+3})$, 
 \begin{equation}\label{eq:ScalInvariance}
  \Scal (w\circ \Lambda ) = \lambda^2 \Scal( w).
\end{equation}
In particular,
 \begin{equation}\label{eq:SimplifiedScalOne}
    \sigma\Scal_1 =\max \Set{ \frac{\sigma \Scal(v_{\theta})}{\abs{\SSS^3}^3} | \theta\in\SSS^1} .
 \end{equation}
\end{prop}
\begin{proof}
The property~\eqref{eq:ScalInvariance} follows from the commutation property~\eqref{eq:GeneralAntiboxCommute} of~$\antibox$. To conclude it suffices to note that, by Proposition~\ref{prop:StriMaxCharacterize}, if $v=S\gbold$ with $\gbold\in\Mrom$ and $\norm{\gbold}_{\Hcaldot^{1/2}}=1$, then $v=\abs{\SSS^3}^{-{1/2}} v_\theta\circ \Lambda$ for a $\theta\in\SSS^1$ and a transformation $\Lambda$ with~$\lambda=1$.
\end{proof}

To compute the maximum in~\eqref{eq:SimplifiedScalOne} we will use the Penrose transform; see~\cite[Appendix A.4]{Hor97}. For this we introduce two coordinate systems on the Minkowski spacetime $\R^{1+3}$ and another two on the curved spacetime $\R\times \SSS^3$, 
where \begin{equation}\label{eq:sphere_definition}
	\SSS^3 = \Set{ (X_0, X_1, X_2, X_3) \in \R^4 : X_0^2+X_1^2+X_2^2+ X_3^2=1}.
\end{equation}
 On $\R^{1+3}$, letting $t\in\R$ be the time coordinate, we define the \emph{polar coordinates} by setting  
\begin{equation}\label{eq:polar}
  \begin{array}{cc}
    r=\abs{x}, & \omega = \frac{x}{\abs{x}}\in \SSS^2.
  \end{array}
\end{equation}
On the other hand,  we define the \emph{light-like coordinates} on $\R^{1+3}$ as 
\begin{equation}\label{eq:lightlike}
 \begin{array}{ccc}
    \xminus=t-r, & \xplus=t+r, & \xminus\le \xplus.
 \end{array}
\end{equation}
On $\R\times \SSS^3$, letting $T$ be the time coordinate,  we define the \emph{spherical polar coordinates} via the equations
\begin{equation}\label{eq:sphere_polar_coordinates}
	\begin{array}{ccc}
	X_0=\cos(\RPolar),& (X_1, X_2, X_3)= \sin(\RPolar)\, \omega , & \ \omega \in \SSS^{2}, \ \RPolar\in [0, \pi].
	\end{array}
\end{equation}
Finally, we define the \emph{light-like coordinates} on $\R\times \SSS^3$ as 
\begin{equation}\label{eq:curved_lightlike_coords}
 \begin{array}{cc}
   \Xminus=\frac12(T-R) , & \Xplus=\frac12(T+R).
 \end{array}
\end{equation}

We can now define an injective map 
\begin{equation}\label{eq:PenroseMapRelation} 
	\begin{array}{cc}
	\Pcal\colon \R^{1+3}\to \R\times \SSS^3, & (T, \cos R, \sin(R)\omega)=\Pcal(t, x), 
	\end{array}
\end{equation}
via the equations
\begin{equation}\label{eq:Penrose_map}
	\begin{array}{cc}
		\Xminus=\arctan \xminus, & \Xplus=\arctan \xplus,
	\end{array}
\end{equation}
remarking that $\Xminus$ and $\Xplus$ take values in
\begin{equation}\label{eq:image_penrose_lightlike}
  \Tcal:=\Set{ (\Xminus, \Xplus)\in [-\tfrac\pi2,\tfrac\pi2]^2 | \Xminus\le \Xplus}. 
\end{equation}
So, the map $\Pcal$ is not surjective and its image $\Pcal(\R^{1+3})$ is 
\begin{equation}\label{eq:image_of_penrose}
	\Pcal(\R^{1+3}) = \Set{ \Big(T, \cos R, \sin (R)\,\omega \Big)\in\R\times \SSS^3 | \begin{array}{c} 
		-\pi<T<\pi \\  
		0\le R< \pi-\abs{T} \\  
		\omega\in \SSS^{2}
		\end{array}
		};
\end{equation}
see the forthcoming Figure~\ref{fig:Penrose}. The map $\Pcal$ is conformal in the sense that
\begin{equation}\label{eq:conformal_transformation}
	dT^2-dR^2-\sin^2 R\,d\omega^2 = \Omega^2\Tonde{ dt^2 - dr^2-r^2d\omega^2},
\end{equation}
where $d\omega^2$ denotes the metric tensor of $\SSS^2$ and the conformal factor $\Omega$ is the scalar field given by
\begin{equation}\label{eq:penrose_conformal_factor}
	\Omega = 2(1+(\xplus)^2)^{-{1/2}}(1+(\xminus)^2)^{-{1/2}} = 2\cos \Xplus \cos \Xminus,
\end{equation}
where the change of variable~\eqref{eq:Penrose_map} is implicit. From now on we omit this change of variable without further specification. 

If $v$ is a scalar field on $\R^{1+3}$, we define a scalar field $V$ on $\Pcal(\R^{1+3})$ by the equation 
\begin{equation}\label{eq:Penrose_field_transformation}
 v=\Omega V, 
\end{equation}
which implies that, at $t=0$ (corresponding to $T=0$), 
\begin{equation}\label{eq:Penrose_field_initial_data}
 \begin{array}{cc}
  \left.v\right|_{t=0} = \left.(\Omega V)\right|_{T=0}, & \left.\partial_tv\right|_{t=0}= \left.(\Omega^{2}\partial_T V)\right|_{T=0}.
 \end{array}
\end{equation}

The scalar field $V$ is called the \emph{Penrose transform} of $v$. We remark that~$v$ is radially symmetric if and only if $V$ depends only on $\Xminus,\Xplus$, and in this case, using \eqref{eq:Penrose_field_transformation} and~\eqref{eq:Penrose_map},   we obtain 
\begin{equation}\label{eq:dAlembert_transf}
\begin{split}
 r\Box\, v &= (\partial_t^2-\partial_r^2)(rv) \\ 
 &= \Omega^2\partial_{\Xplus}\partial_{\Xminus}(r\Omega V) \\ 
 &= \Omega^2\partial_{\Xplus}\partial_{\Xminus}(\sin(R) V),
 \end{split}
\end{equation}
where we used the formula $r\Omega=\sin R$, which can be immediately obtained from \eqref{eq:conformal_transformation} by comparing the factors of $d\omega^2$.

The Penrose transform is relevant in our context, because applying it to~$\vboldstartheta$, as defined in \eqref{eq:ustartheta_fboldtheta}, we obtain a simple expression;
\begin{equation}\label{eq:Ustar_is_cosine}
	\begin{array}{ccc}
	\left.V_{\theta}\right|_{T=0}=\cos \theta ,& \left.\partial_T V_{\theta}\right|_{T=0}=\sin \theta, & \text{and } V_{\theta}=\cos \Tonde{T - \theta}.
	\end{array}
\end{equation}
\begin{prop}\label{prop:SustarPoly}
 It holds that 
 \begin{equation}\label{eq:SustarPoly}
    \Scal(v_\theta) = \frac{\pi^3}{128}\Tonde{ 24\cos^2\theta + 5}.
 \end{equation}
\end{prop}
\begin{proof}
Let $w_\theta:=\antibox(\vstartheta^3)$. Applying the Penrose transform~\eqref{eq:Penrose_field_transformation} to the integral~\eqref{eq:ScalFunctional} that defines $\Scal$, we obtain
\begin{equation}\label{eq:SustarPenrose}
\begin{split}
 \Scal(\vstartheta)&=\iint_{\Pcal(\R^{1+3})} V^3_{\theta} W_\theta\, dTdS=4\pi\int_{-\pi}^\pi\int_0^{\pi-\abs{T}}\cos^3(T-\theta)W_\theta \sin^2 R\, dTdR,
 \end{split}
\end{equation}
where $dS=\sin^2 R\, dR dS_{\SSS^2}$ denotes the volume element on $\SSS^3$. Here we used that $\Omega^4\,dtdx = dT dS$, which follows from~\eqref{eq:conformal_transformation}. Now the change of variable~\eqref{eq:curved_lightlike_coords} yields
\begin{equation}\label{eq:SustarLightlike}
  \Scal(\vstartheta)=8\pi\int\!\!\int_{\Tcal}\cos^3(\Xplus + \Xminus - \theta)\sin(\Xplus-\Xminus)\Wtilde_\theta\, d\Xminus d\Xplus,
\end{equation}
where 
\begin{equation}\label{eq:Wtildetheta}
    \Wtilde_\theta:=\sin(R)W_\theta,
\end{equation}
and $\Tcal$ is the half-square defined in~\eqref{eq:image_penrose_lightlike}. We will prove that $$\Wtilde_\theta(\Xplus, \Xminus)=-\Wtilde_\theta(\Xminus, \Xplus),$$ so that the integrand of~\eqref{eq:SustarLightlike} is symmetric under permutation of the variables, allowing us to consider the integral over the full square~$[-\tfrac\pi2, \tfrac\pi 2]^2$.

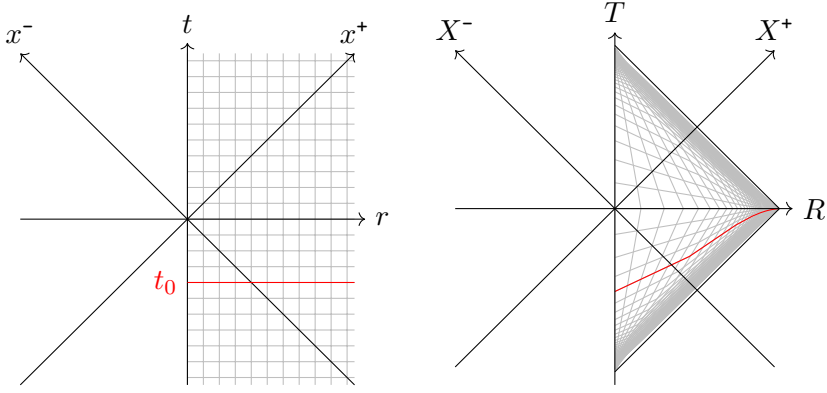
\begin{figure}
\flushleft
\begin{tabular}{cc}
\begin{tikzpicture}[scale=0.7]
\draw[step=0.3, gray!50!white] (0, -3.14) grid (3.14, 3.14);
\draw[->] (-3.14, 0) -> (3.34,0) node[anchor=west] {$r$};
\draw[->] (0, -3.14) -> (0, 3.34) node[anchor=south]{$t$};
\draw[->] (-3.14, -3.14) -- (3.14, 3.14) node[anchor=south] {$\xplus$}; 
\draw[<-] (-3.14, 3.14) node[anchor=south] {$\xminus$} -- (3.14, -3.14);
\draw[red] (0, -1.2) node[anchor=east] {$t_0$}-- (3.14, -1.2) node[anchor=west]{} ;
\end{tikzpicture}

&

\begin{tikzpicture}[scale=0.7]
\foreach \t in {-8, -7.75, ..., 8}
    \draw[gray!50!white, domain=0:30] plot({rad(atan(\t+\x)-atan(\t-\x))}, {{rad(atan(\t+\x)+atan(\t-\x))}});
\foreach \r in {0, 0.25, ..., 5}
	\draw[gray!50!white, domain=-20:20] plot({rad(atan(\x+\r)- atan(\x-\r))}, {{rad(atan(\x+\r)+ atan(\x-\r))}});
	 \draw[red, domain=0:30] plot({rad(atan(-1+\x)- atan(-1-\x))}, {{rad(atan(-1+\x) + atan(-1-\x))}});
\draw[->] (-3, 0) -> (3.34,0) node[anchor=west] {$R$};
\draw[->] (0, -3.34) -> (0, 3.34) node[anchor=south]{$T$};
\draw[->] (-3, -3) -- (3, 3) node[anchor=south] {$\Xplus$}; 
\draw[<-] (-3, 3) node[anchor=south] {$\Xminus$} -- (3, -3);
\draw (0, -pi+0.05) 
-- (pi-0.05, 0) 
-- (0, pi-0.05) 
;
\end{tikzpicture}
\end{tabular}

\caption{As $t_0\to-\infty$, the image under the Penrose map~$\Pcal$ of the hypersurface $t=t_0$ converges uniformly to the hypersurface $\Xminus=-\pi/2$.}
\label{fig:Penrose}
\end{figure}

We compute $\Wtilde_\theta$ explicitly. From the definition of $\antibox$ it follows that
\begin{equation}\label{eq:wtheta_eqn}
 \begin{cases} 
  r\Box w_\theta = r\vstartheta^3, & \text{ on }\R^{1+3}, \\ 
  \Ds \lim_{t\to -\infty} \norm{\wbold_{\theta}}_{\Hcaldot^{1/2}} =0,
 \end{cases}
\end{equation}
and using~\eqref{eq:Penrose_field_transformation}, \eqref{eq:dAlembert_transf}, and the aforementioned formula $r\Omega=\sin R$, we obtain
\begin{equation}\label{eq:TransformEquation}
  \begin{array}{cc}
  r\Box w_\theta = \Omega^2\partial_{\Xplus}\partial_{\Xminus}(\sin(R) W_\theta), &\text{and } r\vstartheta^3=\Omega^2\sin(R) V_{\theta}^3,
  \end{array}
\end{equation}
so the factors of $\Omega^2$ simplify and we obtain from~\eqref{eq:wtheta_eqn} the differential equation
\begin{equation}\label{eq:Wtheta_diff_eqn}
 \partial_{\Xplus}\partial_{\Xminus}\Wtilde_\theta=\sin(\Xplus-\Xminus) \cos^3(\Xplus + \Xminus -\theta).
\end{equation}
The general solution $ \Wtilde_\theta$ of this can be written 
\begin{equation}\label{eq:Wtheta_general}
  \begin{split}
\int_{-\frac\pi2}^{\Xminus}\int_{-\frac\pi2}^{\Xplus} \sin(Z-Y)\cos^3(Y+Z-\theta)\,dYdZ
 +F(\Xplus)+G(\Xminus), 
 \end{split}
\end{equation}
where $F$ and $G$ are arbitrary smooth functions. 

We claim that 
\begin{equation}\label{eq:FplusGvanish}
F(\Xplus)+G(\Xminus)\equiv 0.
\end{equation}
To prove this, we notice that for each fixed $t_0\in\R$, the hypersurface of $\R^{1+3}$ of equation $t=t_0$ is mapped by $\Pcal$ to the hypersurface of equations
\begin{equation}\label{eq:t_slice}
 \begin{array}{cc}
  \Xminus=\arctan(t_0-r), & \Xplus=\arctan(t_0+r), 
 \end{array}
\end{equation}
(see Figure~\ref{fig:Penrose}), which, as $t_0\to-\infty$, converges uniformly to the hypersurface $\Xminus=-\pi/2$. The condition $\norm{\wbold_\theta(t)}_{\Hcaldot^{1/2}}\to 0$ thus implies that $$\Wtilde_\theta|_{\Xminus=-\frac\pi2}=0.$$ We obtain another condition by observing that, since $w_\theta$ is smooth and radially symmetric, the function $W_\theta$ must be regular at $R=0$, which implies that $\Wtilde_\theta|_{R=0}=0$. Now the integral of~\eqref{eq:Wtheta_general} satisfies both conditions. The first one is obvious, while the second follows from symmetry, since 
$$\Xminus|_{R=0}=\Xplus|_{R=0},$$ 
so the domain of integration is symmetric under permutation of the variables $Y, Z$, while the integrand function changes sign. This proves~\eqref{eq:FplusGvanish}. 

Returning to~\eqref{eq:SustarLightlike}, the fact that $\Wtilde_\theta(\Xplus, \Xminus)=-\Wtilde_\theta(\Xminus, \Xplus)$ is immediate from the explicit form of $\Wtilde_\theta$. Thus the integral in~\eqref{eq:SustarLightlike} can be replaced by the integral over $[-\tfrac\pi2, \tfrac\pi2]^2$, with a multiplicative factor of $\tfrac12$. More precisely, letting
\begin{equation}\label{eq:F_coscube_sin}
  F(Y, Z, \theta):=\sin(Z-Y)\cos^3(Y+Z-\theta),
\end{equation}
we have the formula
\begin{equation}\label{eq:SustarLightlikeProcessed}
 \Scal(\vstartheta)=4\pi\int_{-\frac\pi2}^\frac{\pi}{2} \int_{-\frac{\pi}{2}}^{\frac{\pi}{2}}\int_{-\frac\pi2}^{\Xminus} \int_{-\frac\pi2}^{\Xplus} F(\Xminus, \Xplus, \theta)F(Y, Z, \theta)\, d\Xminus d\Xplus dY dZ,
\end{equation}
which allows for explicit computation, yielding~\eqref{eq:SustarPoly}.

\end{proof}                                                                                                                                                                                                                                                                                                                

Combining Propositions~\ref{prop:SoneIsSonetheta} and~\ref{prop:SustarPoly} we obtain the value of the constant.

\begin{cor}\label{cor:ScalThetaMax}
  The function $\sigma \Scal(v_\theta)$ attains its maximum for $\theta$ of the form $k\pi$ when $\sigma >0$, and for $\theta$ of the form $\frac\pi 2+ k\pi$ when $\sigma <0$; here $k\in\Z$. The constant $\Scal_1$ in Theorem~\ref{thm:MainThree} can be written
  \begin{equation}\label{eq:ScalOne_explicit}
    \Scal_1=\begin{cases}\Ds  \frac{\Scal(v_{0})}{\abs{\SSS^3}^3}=\frac{29}{128}\Big(\frac{\pi}{\abs{\SSS^3}}\Big)^3, & \sigma >0, \\ \Ds
             \frac{\Scal(v_{\pi/2})}{\abs{\SSS^3}^3}=\frac{5}{128}\Big(\frac{\pi}{\abs{\SSS^3}}\Big)^3, & \sigma<0.
            \end{cases}
  \end{equation}
  \end{cor}
 \begin{rem}\label{rem:goursat}
  In the proof of Proposition~\ref{prop:SoneIsSonetheta} we solved a boundary value problem for the wave equation on $\R\times \SSS^3$ with data on a light cone. This is known in the literature as \emph{Goursat problem}; see~\cite{Hor90,BaSeZh90}.
  \end{rem}

\section{Existence of maximizers}
We follow the lines of~\cite[Section~2]{DuMeRo11} to show that the supremum~\eqref{eq:I_functional} is attained for small enough values of $\delta$. We recall from Proposition~\ref{prop:small_data_theory} that $\Phi(\fbold)=u$ denotes the solution to the fixed point equation associated to~\eqref{eq:cubic_NLW}
\begin{equation}\label{eq:FixPtEqn}
  u=S\fbold + \sigma \antibox(u^3),
\end{equation}
provided that such a solution exists and is unique. We require in this section the concentration-compactness tools of Appendix~\ref{app:ConcComp}.

\begin{lem}\label{lem:ConcComp}
Suppose that $\delta>0$ satisfies
 \begin{enumerate}
  \item Scattering:   $I(\delta)<\infty$;
  \item Superadditivity: for all $\alpha\in(0, \delta)$, 
  \begin{equation}\label{eq:Superadditivity}
   I(\sqrt{\delta^2-\alpha^2})+I(\alpha)<I(\delta);
  \end{equation}
  \item Upper semicontinuity: for any sequence $\alpha_n\le \delta$, 
  \begin{equation}\label{eq:UpSemiCont}
    \limsup_{n\to\infty} I(\alpha_n)\le I(\limsup_{n\to\infty} \alpha_n).
  \end{equation}
 \end{enumerate}
  Then there exists a solution $u$ to~\eqref{eq:cubic_NLW} such that 
  \begin{equation}\label{eq:uA_maximizes}
  \begin{array}{ccc}
  \Ds   \lim_{t\to-\infty}\norm{\ubold(t)}_{\Hcaldot^{1/2}}=\delta &\text{ and }& \Ds \norm{u}_{L^4(\R^{1+3})}^4= I(\delta).
  \end{array}
  \end{equation}
\end{lem}

\begin{proof}
 Let $u_n$ be a maximizing sequence of $I$, that is
 \begin{equation}\label{eq:un_maxim}
    \begin{array}{ccc}
      u_n=\Phi(\fbold_n), &\ \norm{\fbold_n}_{\Hcaldot^{1/2}}\le\delta, &\ \Ds I(\delta)=\lim_{n\to \infty} \norm{u_n}_{L^4}^4.
    \end{array}
 \end{equation}
  We consider a profile decomposition~\eqref{eq:profile_decomposition} of the sequence $\fbold_n$, and we claim that all profiles $\{\Fbold^j\ :\ j\in\N_{\ge 1}\}$ vanish but one. 
  
  To prove this, we denote by $\gbold_n$ the sequence obtained by subtracting the profile $\Fbold^j$ from $\fbold_n$, that is
  \begin{equation}\label{eq:gbold_n}
    \gbold_n=\fbold_n - \left.\lambda^{(j)}_n S\Fbold^j\circ \Lambda^j_n\right|_{t=0},
  \end{equation}
  and we construct the corresponding solution $W_n=\Phi(\gbold_n)$. By the nonlinear profile decomposition, Corollary~\ref{cor:nonlinear_profdecomp}, we have that
  \begin{equation}\label{eq:un_decomp}
    u_n(t, x)= \lambda^{(j)}_n U^j(\Lambda^j_n(t, x)) + W_n(t, x) + h_n(t, x), 
  \end{equation}
  where $\norm{h_n}_\Lfour+\sup_{t\in\R} \norm{\hbold_n(t)}_{\Hcaldot^{1/2}} \to 0$ as $n\to\infty$. By the Pythagorean expansion~\eqref{eq:Pythagorean_energy} of the $\PhaseSpace$ norm, 
  \begin{equation}\label{eq:A_F_gn}
    \delta^2\ge\norm{\fbold_n}^2_{{\Hcaldot^{1/2}}} = \norm{\Fbold^j}_{{\Hcaldot^{1/2}}}^2 + \norm{\gbold_n}_{\Hcaldot^{1/2}}^2 +o(1), 
  \end{equation}
  and by Remark~\ref{rem:LfourOrtho}, 
  \begin{equation}\label{eq:LfourWn}
    \norm{u_n}_{L^4}^4 = \norm{ U^j}_{L^4}^4  + \norm{W_n}_{L^4}^4+ o(1). 
  \end{equation}
  Since $u_n$ is a maximizing sequence, we infer from~\eqref{eq:A_F_gn} and~\eqref{eq:LfourWn} 
  \begin{equation}\label{eq:FinishNoDichot}
    \begin{split}
    I(\delta)&=\norm{U^j}_{L^4}^4 + \limsup_{n\to\infty} \norm{W_n}_{L^4}^4\\ &\le I\big(\norm{\Fbold^j}_{\Hcaldot^{1/2}}\big)+I\Big(\sqrt{\delta^2-\norm{\Fbold^j}_{\Hcaldot^{1/2}}^2}\Big), 
    \end{split}
  \end{equation}
  where we also used the upper semicontinuity property~\eqref{eq:UpSemiCont}. Now, the superadditivity property~\eqref{eq:Superadditivity} implies that 
  \begin{equation}\label{eq:vanishing_or_compactness} 
  \begin{array}{rcrc}
  \text{either} &
   \norm{\Fbold^j}_{\Hcaldot^{1/2}}=0, & 
   \text{or} &
   \norm{\Fbold^j}_{\Hcaldot^{1/2}}=\delta.
   \end{array}
  \end{equation}
  It cannot be that $\Fbold^j=\obold$ for all $j\ge 1$, for otherwise the nonlinear profile decomposition~\eqref{eq:nonlinear_profdecomp} would give the contradiction $I(\delta)=0$. On the other hand, if $\lVert \Fbold^j\rVert_{\Hcaldot^{1/2}}=\delta$ then, by~\eqref{eq:A_F_gn}, $\norm{\gbold_n}_{\Hcaldot^{1/2}}\to 0$ as $n\to \infty$, which means that $\Fbold^k=\obold$ for all $k\ne j$. 
  
  We have thus proven that there exists one and only one nonvanishing profile $\Fbold$ for the sequence $\fbold_n$. Letting $U$ denote the corresponding nonlinear profile, Corollary~\ref{cor:nonlinear_profdecomp} implies that $I(\delta)=\norm{U}_{L^4}^4$, and the proof is complete.
  \end{proof}                                                                                                                                                                                                                                                                                                                
  We now turn to the proof that, if $\delta>0$ is sufficiently small, then the three properties of Lemma~\ref{lem:ConcComp} are satisfied. We already dealt with the first one in Proposition~\ref{prop:small_data_theory}. The following lemma implies the third property and will also be used in the proof of the second property.
  \begin{lem}\label{lem:sub_differentiable}
    There exists $A, C_1, C_2>0$ such that
    \begin{equation}\label{eq:I_sub_diff}
      C_1\abs\eps \delta^3 \le \abs{I(\delta+ \eps) - I(\delta)}\le C_2\abs\eps \delta^3,\quad \forall\, \eps\in(-\delta/2, \delta/2),
    \end{equation}
  whenever $\delta\in(0, A]$.
In particular, $I$ is continuous on $(0, A/2]$.
  \end{lem}
  \begin{proof}
  In fact we will prove that 
   \begin{equation}\label{eq:asympt_sub_diff}
      4\Scal_0 \eps \delta^3+O(\eps\delta^5)\le I(\delta+\eps)-I(\delta)\le 4\Scal_0\eps(\delta+\eps)^3+O(\eps\delta^5),
   \end{equation}
   from which~\eqref{eq:I_sub_diff} follows by taking $A>0$ sufficiently small. For this we let $\norm{\fbold}_{\Hcaldot^{1/2}}=\delta$ and $u=\Phi(\fbold)$ be close to maximal in the sense that 
   \begin{equation}\label{eq:approx_maximiz}
      \begin{array}{ccc}
	I(\delta)-\norm{u}_{L^4}^4=O(\eps\delta^5),
      \end{array}
   \end{equation}
   and we define
   \begin{equation}\label{eq:v_deltaeps_vtilde}
      \begin{array}{cc}
       u_{\eps}:=\Phi\Tonde{ (1+\tfrac\eps\delta) \fbold}, & \utilde_{\eps}:=(1+\tfrac\eps\delta)u.
      \end{array}
   \end{equation}
   With these definitions, since $\Box\, u+\sigma u^3=0$, we have that
   \begin{equation}\label{eq:approx_eq_vtildeeps}
     \sigma e:= \Box\,\utilde_{\eps} - \sigma \utilde_\eps^3 = -2\sigma\tfrac\eps\delta u^3 + O(\tfrac{\eps^2}{\delta^2} u^3),
   \end{equation}
  where the big-O symbol refers to the $L^{4/3}(\R^{1+3})$ norm, and since $\norm{u}_{L^4}$ is $O(\delta)$, we can conclude that $$\norm{e}_{L^{4/3}}=O(\eps\delta^2).$$ Moreover, it is clear that $\norm{\ubold_\eps(t)-\utildebold_\eps(t)}_\PhaseSpace \to 0$ as $t\to-\infty$, and so we can apply the forthcoming perturbation Lemma~\ref{lem:perturb} to obtain $$\norm{u_{\eps}-\utilde_{\eps}}_{L^4}\le C \eps \delta^2,$$
and we infer that 
   \begin{equation}\label{eq:v_deltaeps_vtilde_deltaeps}
      \norm{ u_{\eps}}_{L^4}^4 = \norm{ \utilde_{\eps}}_{L^4}^4 + O(\eps \delta^5),
   \end{equation}
   where the constant implicit in the big-O notation depends on $A$ only. 
   
   We now insert~\eqref{eq:v_deltaeps_vtilde_deltaeps} into the inequality $I(\delta+\eps)\ge \norm{u_{\eps}}_{L^4}^4$, which follows from the definition of $I$. We obtain 
   \begin{equation}\label{eq:Ideltaeps_dominates}
    \begin{split}
      I(\delta+\eps) &\ge (1+\tfrac\eps\delta)^4\norm{ u}_{L^4}^4 + O(\eps\delta^5) \\ 
      &\ge I(\delta) + 4\tfrac\eps\delta I(\delta) +O(\eps\delta^5),
    \end{split}
   \end{equation}
   where we used the elementary inequality $(1+\tfrac\eps\delta)^4\ge 1+4\tfrac\eps\delta$ and the closeness condition \eqref{eq:approx_maximiz}. Now by the asymptotic Proposition~\ref{prop:SecondOrder}, we know that $I(\delta)=\Scal_0\delta^4+O(\delta^6)$ which can be inserted to complete the proof of the first inequality in~\eqref{eq:asympt_sub_diff}.

   To prove the second inequality and complete the proof of Lemma~\ref{lem:sub_differentiable}, we let $\norm{\fbold}_{\Hcaldot^{1/2}}=\delta+\eps$ and $u=\Phi(\fbold)$ be close to maximal in the sense that 
   \begin{equation}\label{eq:udelta_eps_analog}
   \begin{array}{cc}
    I(\delta+\eps)-\norm{ u}_{L^4}^4 =O(\eps(\delta+\eps)^5).
    \end{array}
    \end{equation}
    Then we define $u_\eps:=\Phi( (1-\tfrac\eps{\delta+\eps})\fbold)$ and $\utilde_\eps:=(1-\tfrac\eps{\delta+\eps})u$, and argue as before.
  \end{proof}                                                                                                                                                                                                                                                                                                                
  \begin{prop}\label{prop:superadditivity}
    For sufficiently small $\delta>0$,  
    \begin{equation}\label{eq:superadd_two}
      I(\alpha)+I(\sqrt{\delta^2-\alpha^2})<I(\delta) \quad \forall\  \alpha\in(0, \delta).
    \end{equation}
  \end{prop}
  \begin{proof}
   This follows from the fact that $I$ is a super-additive function of $\delta$ to main order, because $I(\delta)=\Scal_0\delta^4+O(\delta^6)$, together with the estimates of  Lemma~\ref{lem:sub_differentiable}, which rule out excessive fluctuations; see~\cite[Proposition~2.7]{DuMeRo11}.
  \end{proof}                                                                                                                                                                                                                                                                                                                
  
  \section{Conditional uniqueness of maximizers}


If $u=\Phi(\fbold)$ is a maximizer to $I(\delta)$, and
\begin{equation}\label{eq:PoincUniq}
	\begin{array}{cc}
		\Lambda(t, x)= L^{\beta}\big(\lambda (t-t_{0}), \lambda (x-x_{0})\big),& \lambda>0, \abs\beta < 1, t_0\in \R, x_0\in \R^3, 
	\end{array}
\end{equation}  
then $\lambda (u\circ \Lambda)$ is again a maximizer to $I(\delta)$; this is an immediate consequence of Theorem~\ref{thm:NonlinearPoincareAction}. In this section we give a partial result about the problem of uniqueness of maximizers, up to this transformation. The main tool is the forthcoming Lemma~\ref{lem:transv_nondegenerate}, which is a local version of the sharpened Strichartz estimate of Lemma~\ref{lem:SharpenedStrichartz}.

We begin by showing that each maximizer of $I(\delta)$ has a unique metric projection on the manifold $\Mrom$ of linear maximizers. 
We refer to Appendix~\ref{app:geometry} for the definition of the tangent space $T_{\fboldstar} \Mrom$. 
\begin{lem}\label{lem:projecting_nonlinear_maximizers}
	Let $u=\Phi(\fbold)$ be such that $\norm{u}_{L^4(\R^{1+3})}^4=I(\delta)$. If $\delta>0$ is sufficiently small, then there exists a unique $\fboldstar\in\Mrom\setminus\{\obold\}$ such that 
	\begin{equation}\label{eq:fboldstar_projects}
		\norm{\fbold-\fboldstar}_\PhaseSpace=\dist(\fbold, \Mrom). 
	\end{equation}
	Moreover, $\fbold-\fboldstar\,\bot\, T_{\fboldstar }\Mrom$, where $\bot$ denotes orthogonality with respect to the~$\PhaseSpace$ scalar product.
\end{lem}
\begin{proof} This is proved in Appendix~\ref{app:geometry}, the main issue being uniqueness.  Lemma~\ref{lem:ConcComp} ensures that $\norm{\fbold}_\PhaseSpace=\delta$, while by Proposition~\ref{prop:SecondOrder}, we have $\dist(\fbold, \Mrom)=O(\delta^3)$. Thus, if $\delta$ is sufficiently small, then the forthcoming Proposition~\ref{prop:Metric_Projection} can be applied.
\end{proof}
The elements $\fboldstar$ of $\Mrom\setminus\{\obold\}$ have the unique representation 
\begin{equation}\label{eq:fboldstar_parameterized}
	\fboldstar=\left.\delta \lambda \vbold_\theta\circ\Lambda\right|_{t=0}, 
\end{equation}
where $\vbold_\theta=(v_\theta, \partial_t v_\theta)$ are particular solutions to the linear wave equation, as defined in \eqref{eq:ustartheta_fboldtheta} in the fifth section; see Appendix~\ref{app:geometry}. We let  $\theta(\fboldstar)$ denote the unique $\theta\in \SSS^1$. We recall that this parameter $\theta$ does not correspond to any symmetry of~\eqref{eq:cubic_NLW}; see Remark~\ref{rem:Ph_no_symmetry}.

We can now state the result.
  \begin{thm}\label{thm:Uniqueness}
  	Suppose that  $u_{\fbold}=\Phi(\fbold)$ and $u_{\gbold}=\Phi(\gbold)$ satisfy 
\begin{equation}\label{eq:two_maximizers}
		\begin{array}{cc}
		\norm{\fbold}_{\Hcaldot^{1/2}}=\norm{\gbold}_{\Hcaldot^{1/2}}=\delta, &\Ds \text{and}\quad I(\delta)=\norm{u_{\fbold}}_{L^4}^4  = \norm{u_{\gbold}}_{L^4}^4,
		\end{array}
	\end{equation} 
	with $\delta$ sufficiently small. Suppose moreover that the unique projections $\fboldstar$ and $\gbold_{\star}$ satisfy 
	\begin{equation}\label{eq:theta_equality}
		\theta(\fboldstar)=\theta(\gbold_{\star}).
	\end{equation}
	Then there is a transformation $\Lambda$ of the form~\eqref{eq:PoincUniq} such that $u_{\gbold}=\lambda(u_{\fbold}\circ\Lambda)$. 
\end{thm}
\begin{rem}\label{rem:uniqueness} 
The assumption~\eqref{eq:theta_equality} makes this uniqueness result conditional. We conjecture that such an assumption is not necessary; that there is a single $\theta(\fboldstar)$ for each maximizer $\fbold$ to $I(\delta)$. Indeed, by Proposition~\ref{prop:SecondOrder}, such $\theta(\fboldstar)$ must be close to a maximum of the function $\sigma \Scal(v_\theta)$. By Corollary~\ref{cor:ScalThetaMax}, such maxima differ by an integer multiple of $\pi$, and so correspond to just one linear maximizer $\fbold_\theta$, up to a sign.
\end{rem}

\begin{lem}[Lemma~5.1 of \cite{Negro18}]\label{lem:transv_nondegenerate}
  Let $\psi$ be the  functional defined by
   \begin{equation}\label{eq:psi_funct}
      \psi(\fbold):=\Scal_0\norm{\fbold}_{\Hcaldot^{1/2}}^4-\norm{S\fbold}_{L^4(\R^{1+3})}^4.
   \end{equation}
   Then there exists $C>0$ such that, for all $\mbold \in \Mrom\setminus\{\obold\}$, 
   \begin{equation}\label{eq:Hessian_nondegenerate}
    \begin{array}{cc}\Ds
      \left.\frac{d^2}{d\eps^2} \psi(\mbold + \eps \mboldbot)\right|_{\eps=0}\ge C\norm{\mbold}_{\Hcaldot^{1/2}}^2 \norm{\mboldbot}_{\Hcaldot^{1/2}}^2,& \forall\ \mboldbot \bot T_{\mbold}\Mrom.
    \end{array}
   \end{equation}
  \end{lem}

The derivative in~\eqref{eq:Hessian_nondegenerate} can be computed to be
\begin{align}
		\label{eq:psi_second}
		\!\!\!\!\!\!\!\frac12\left.\frac{d^2}{d\eps^2}\psi(\mbold+\eps\mboldbot)\right|_{\eps=0} \!\!\!\!\!\!\!&= 2\Scal_0\norm{\mbold}_{\Hcaldot^{1/2}}^2\norm{\mboldbot}_{\Hcaldot^{1/2}}^2 -6\iint_{\R^{1+3}}\!(S\mbold)^2(S\mboldbot)^2;
\end{align}
see the proof of Lemma~5.1 of \cite{Negro18} for more details.

  \begin{proof}[Proof of Theorem~\ref{thm:Uniqueness}]
   By the unique representation \eqref{eq:fboldstar_parameterized}, our assumption~\eqref{eq:theta_equality}, and Lemma~\ref{lem:projecting_nonlinear_maximizers}, up to changing $u_{\fbold}$ with $\lambda(u_{\fbold} \circ \Lambda)$ and $u_{\gbold}$ with $\lambda'(u_{\gbold}\circ\Lambda')$, where $\Lambda$ and $\Lambda'$ are transformations of the form~\eqref{eq:PoincUniq}, we can decompose 
   \begin{equation}\label{eq:fdelta_gdelta}
   	\begin{array}{cccc}
		\fbold=c\delta\mbold + \fboldbot, & \gbold=c'\delta\mbold+\gboldbot, &\text{with }\fboldbot\bot T_\mbold \Mrom \text{ and } \gboldbot\bot T_{\mbold}\Mrom,
	\end{array}
\end{equation}
where $\mbold=\abs{\SSS^3}^{-1/2}\fbold_{\theta(\fboldstar)}$, so that $\norm{\mbold}_{\Hcaldot^{1/2}}=1$.  We denote 
\begin{equation}\label{eq:h_hboldbot}
		\hbold:=\fbold - \gbold, \quad \text{and}\quad \hboldbot:=\fboldbot-\gboldbot.
\end{equation}
The proof will be complete once we show that $\hbold=\obold$. 

We now record the necessary estimates. First, we recall from Proposition~\ref{prop:SecondOrder} that 
\begin{equation}\label{eq:fboldbot_gboldbot_distance}
	\begin{array}{cc}
		\norm{\hboldbot}_{\Hcaldot^{1/2}}\le \dist(\fbold, \Mrom)+\dist(\gbold, \Mrom)=O(\delta^3).
	\end{array}
\end{equation}
Now using the orthogonality, we can expand the identity $\norm{\fbold}^2_{\Hcaldot^{1/2}}=\norm{\gbold}^2_{\Hcaldot^{1/2}}$, to obtain 
\begin{equation}\label{eq:c_c'_controlled_bot}
	\delta^2\abs{c^2-c'^2}=\abs{\norm{\gboldbot}_{\Hcaldot^{1/2}}^2-\norm{\fboldbot}^2_{\Hcaldot^{1/2}}}\le C\delta^3\norm{\hboldbot}_{\Hcaldot^{1/2}},
\end{equation}
so that
\begin{equation}\label{eq:c_minus_c'}
	(c-c')^2=\Tonde{\frac{c^2-c'^2}{c+c'}}^2\le C\delta^2\norm{\hboldbot}_{\Hcaldot^{1/2}}^2.
\end{equation} 
In particular, 
\begin{equation}\label{eq:hdelta_control_hboldbot}
	\norm{\hbold}_{\Hcaldot^{1/2}}^2 = (c-c')^2\delta^2 + \norm{\hboldbot}_{\Hcaldot^{1/2}}^2 =\norm{\hboldbot}_{\Hcaldot^{1/2}}^2+ O(\delta^4\norm{\hboldbot}_{\Hcaldot^{1/2}}^2).
\end{equation}
We now define $w:=u_{\fbold}-u_{\gbold}$; that is, $w=\Phi(\fbold)-\Phi(\gbold)$. By the definition \eqref{eq:u_delta_recall} of $\Phi$, we have that 
\begin{equation}\label{eq:w_delta_fixpoint_eq}
	w = S\hboldbot + S\big( (c-c')\delta\mbold\big) +\sigma \antibox\Tonde{u_{\fbold}^3-u_{\gbold}^3}, 
\end{equation}
and the Strichartz estimates~\eqref{eq:ConfStrichartz} give $$\norm{\antibox\Tonde{u_{\fbold}^3-u_{\gbold}^3}}_{L^4}\le C\delta^2\norm{\hbold}_{\Hcaldot^{1/2}}.$$  Thus by~\eqref{eq:c_minus_c'} and~\eqref{eq:hdelta_control_hboldbot} we have 
\begin{equation}\label{eq:w_delta_asymp}
	w= S\hboldbot + O(\delta^2 \norm{\hboldbot}_{\Hcaldot^{1/2}});
\end{equation}
the big-O symbol referring to the $\Lfour$ norm. Analogously, we see that
\begin{equation}\label{eq:v_delta_asymp}
	u_{\gbold} = S(c'\delta \mbold) + O(\delta^3).
\end{equation}

With these estimates in hand, we may now proceed with the proof. The key step is given by the formula
\begin{equation}\label{eq:UniqKeyStep}
	\norm{u_{\fbold}}_{L^4}^4 - \norm{u_{\gbold}}_{L^4}^4 = -\frac{1}{2}\left.\frac{d^2}{d\eps^2} \psi(c'\delta\mbold + \eps\hbold_\bot) \right|_{\eps=0}+O(\delta^3\norm{\hboldbot}_{{\Hcaldot^{1/2}}}^2),
\end{equation}
which we will prove later. Note that the left-hand side vanishes by assumption. So, once~\eqref{eq:UniqKeyStep} is proven, Lemma~\ref{lem:transv_nondegenerate} will imply that 
\begin{equation}\label{eq:hboldbot_smaller_hboldbot}
	\delta^2\norm{\hboldbot}_{\Hcaldot^{1/2}}^2\le C\delta^3 \norm{\hboldbot}^2_{\Hcaldot^{1/2}},
\end{equation}
for an absolute constant $C>0$, which is only possible if $\norm{\hboldbot}_{\Hcaldot^{1/2}}= 0$, provided that $\delta<C^{-1}$. By~\eqref{eq:hdelta_control_hboldbot}, this would imply that $\hbold=\obold$, concluding the proof.

In order to prove~\eqref{eq:UniqKeyStep}, we recall that $u_{\fbold}=u_{\gbold}+w$ and we expand
\begin{equation}\label{eq:expansion_uniqueness}
\begin{split}
	&\iint_{\R^{1+3}}(u_{\gbold}+w)^4-\iint_{\R^{1+3}}u_{\gbold}^4 =4\iint_{\R^{1+3}}u_{\gbold}^3w + 6\iint_{\R^{1+3}} u_{\gbold}^2w^2+O(\delta\norm{\hboldbot}_{\Hcaldot^{1/2}}^3) \\ 
	&=4\iint_{\R^{1+3}}u_{\gbold}^3 w +6\iint_{\R^{1+3}}(S (c'\delta\mbold))^2(S\hboldbot)^2 +O(\delta^3\norm{\hboldbot}_{\Hcaldot^{1/2}}^2 +\delta\norm{\hboldbot}_{\Hcaldot^{1/2}}^3),
\end{split}
\end{equation}
where we used~\eqref{eq:w_delta_asymp} and~\eqref{eq:v_delta_asymp}. By~\eqref{eq:fboldbot_gboldbot_distance}, we know that $$O(\delta^3\norm{\hboldbot}_{\Hcaldot^{1/2}}^2 +\delta\norm{\hboldbot}_{\Hcaldot^{1/2}}^3)=O(\delta^3\norm{\hboldbot}_{\Hcaldot^{1/2}}^2).$$ Thus, using~\eqref{eq:psi_second}, to conclude the proof of~\eqref{eq:UniqKeyStep} it remains to show that 
\begin{equation}\label{eq:first_order_exp_uniqueness}
	4\iint_{\R^{1+3}} u_{\gbold}^3w = -2\Scal_0c'^2\delta^2\norm{\hboldbot}_{\Hcaldot^{1/2}}^2 +O(\delta^3\norm{\hboldbot}_{\Hcaldot^{1/2}}^2),
\end{equation}
 for which we will use the Lagrange multiplier theorem. 

For $\kbold\in\PhaseSpace(\R^3)$, let 
\begin{equation}\label{eq:BigW_BigG}
	\begin{array}{cc}
		W(\kbold):=\Phi(\gbold+\kbold)-\Phi(\gbold),& G(\kbold):=\norm{\gbold+ \kbold}^2_{{\Hcaldot^{1/2}}},
	\end{array}
\end{equation}
so that $w=W(\hbold)$, $0=W(\obold)$ and $\delta^2= G(\obold)$. Since $u_{\gbold}=\Phi(\gbold)$ is a maximizer for $I(\delta)$, we have that 
\begin{equation}\label{eq:maxim_probl_uniq}
	\iint_{\R^{1+3}}u_{\gbold}^4=\max\Set{ \iint_{\R^{1+3}}(u_{\gbold} + W(\kbold))^4 | G(\kbold)=\delta^2};
\end{equation}
that is, $\kbold=\obold$ is a solution to the constrained optimization problem on the right-hand side of~\eqref{eq:maxim_probl_uniq}. In particular, there exists a Lagrange multiplier $\mu\in\mathbb R$ such that 
\begin{equation}\label{eq:LagrMult}
	\begin{array}{cc}\Ds 
		\mu G'(\obold)\kbold= 4\iint_{\R^{1+3}}u_{\gbold}^3 W'(\obold)\kbold, & \forall \kbold\in\PhaseSpace(\R^3), 
	\end{array}
\end{equation}
where the notation $F'(\obold)\kbold$ denotes the directional derivative $\left.\tfrac{d}{d\eps}F(\eps\kbold)\right|_{\eps=0}$.
We need to compute $\mu$. First we note that, by the definition of $G$,
\begin{equation}\label{eq:diff_G}
	\mu G'(\obold)\kbold= 2\mu \Braket{\gbold|\kbold}_{\Hcaldot^{1/2}}.
\end{equation}
Now, by the definition~\eqref{eq:BigW_BigG} of $W$,
\begin{equation}\label{eq:Wfixpt}
	W(\kbold)=S\kbold + \sigma\antibox\Tonde{\Phi(\gbold+\kbold)^3-\Phi(\gbold)^3},
\end{equation}
and the right-hand side is differentiable; see Remark~\ref{rem:Differentiability}. The directional derivative equals
\begin{equation}\label{eq:Wprime_zero}
	W'(\obold)\kbold= S\kbold+3\antibox (\Phi(\gbold)^2\Phi'(\gbold)\kbold)=S\kbold + O(\delta^2\norm{\kbold}_{\Hcaldot^{1/2}}).
\end{equation}
We insert this, the expansion~\eqref{eq:v_delta_asymp} of $u_{\gbold}$ and the formula $\gbold=c'\delta\mbold+\gboldbot$, into~\eqref{eq:LagrMult} to obtain 
\begin{equation}\label{eq:mu_delta_first}
	2\mu\Braket{c'\delta\mbold|\kbold}_{\Hcaldot^{1/2}}+2\mu\Braket{\gboldbot|\kbold}_{\Hcaldot^{1/2}}=4\iint_{\R^{1+3}} (S(c'\delta \mbold))^3 S\kbold + O(\delta^5\norm{\kbold}_{\Hcaldot^{1/2}}). 
\end{equation}
We evaluate this equation at $\kbold=\mbold$, using that $\Braket{\gboldbot|\mbold}_{\Hcaldot^{1/2}}=0$ and  that  $\norm{S\mbold}_{L^4}^4=\Scal_0$. The result is
\begin{equation}\label{eq:mu_delta_computed}
	\mu=2\Scal_0c'^2\delta^2+O(\delta^5).
\end{equation}

We are now ready to conclude the proof of~\eqref{eq:first_order_exp_uniqueness}.  We notice that $$\norm{\gbold}_{\Hcaldot^{1/2}}^2=\norm{\gbold+\hbold}_{\Hcaldot^{1/2}}^2=\delta^2,$$ so $2\Braket{\gbold|\hbold}_{\Hcaldot^{1/2}}=-\norm{\hbold}_{\Hcaldot^{1/2}}^2$. Using this,
\begin{equation}\label{eq:from_vdelta_to_braket}
	\begin{split}
		4\iint_{\R^{1+3}} u_{\gbold}^3w &= 4\iint_{\R^{1+3}}u_{\gbold}^3W'(\obold)\hbold + O(\delta^3\norm{\hbold}_{\Hcaldot^{1/2}}^2), \\ 
		&=2\mu \Braket{\gbold| \hbold}_{\Hcaldot^{1/2}} +O(\delta^3\norm{\hbold}_{\Hcaldot^{1/2}}^2) \\
		&=-2\Scal_0c'^2\delta^2\norm{\hbold}_{\Hcaldot^{1/2}}^2 + O(\delta^3\norm{\hbold}_{\Hcaldot^{1/2}}^2),
	\end{split}
\end{equation}
where we used that $w=W(\hbold)=W'(\obold)\hbold + O(\norm{\hbold}_{\Hcaldot^{1/2}}^2)$. Since $\norm{\hbold}_{\Hcaldot^{1/2}}$ equals $\norm{\hboldbot}_{\Hcaldot^{1/2}}$ to main order (see~\eqref{eq:hdelta_control_hboldbot}), the proof of~\eqref{eq:first_order_exp_uniqueness} is complete.
  \end{proof}

\appendix
\section{Nonlinear profile decomposition}\label{app:ConcComp}
In this section, we adapt the linear profile decomposition of Ramos (see~\cite{Ramos12}) to sequences of solutions of~\eqref{eq:cubic_NLW}. This is classical, and similar to what is done in~\cite{Ram18}, with the difference that we assign the initial data at $t=-\infty$, in the sense of Proposition~\ref{prop:small_data_theory}.

We consider sequences of transformations of the form 
\begin{equation}\label{eq:Lambda_n}
	\Lambda_n(t, x)=L^{\beta_n} \big( \lambda_n(t-t_n), \lambda_n(x-x_n)\big), 
\end{equation}
where $\lambda_n\in(0, \infty), t_n\in \R, x_n \in \R^3$ and $\beta_n\in\R^3$ with $|\beta_n|<1$. Here we use the notation $a\sim b$, to mean that an absolute constant $C>0$ exists such that $C^{-1} a \le b \le C a$. The following definition is taken from~\cite{Ramos12}.
\begin{defn}\label{def:orthogonality}
	Consider sequences $(\Lambda^1_n)_{n\in\N}, (\Lambda^2_n)_{n\in\N}$ as above and let 
	\begin{equation}\label{eq:rapidity}
		\begin{array}{cc}
		\frac{(\ell^j_n)^2-1}{(\ell^j_n)^2+1}= |\beta^j_n|, & \ell^j_n\in[1, \infty).
		\end{array}
	\end{equation}
	The sequences $\Lambda^1_n$ and $\Lambda^2_n$ are \emph{orthogonal} if at least one of the following properties is satisfied: 
	\begin{enumerate}
	\item Lorentz property:
	\begin{equation}\label{eq:Lorentz_Div}
		\lim_{n\to\infty} \frac{\ell^1_n}{\ell^2_n}+\frac{\ell^2_n}{\ell^1_n}=\infty.
	\end{equation}
	\item Rescaling property:
	\begin{equation}\label{eq:Scaling_Div}
		\lim_{n\to\infty} \frac{\lambda^{(1)}_n}{\lambda^{(2)}_n}+\frac{\lambda^{(2)}_n}{\lambda^{(1)}_n}=\infty.
	\end{equation}
	\item Angular property: it holds that $\lambda^{(1)}_n\sim \lambda^{(2)}_n$, $\ell^1_n\sim \ell^2_n$ and 
	\begin{equation}\label{eq:Angular_Div}
		\lim_{n\to\infty} \ell^1_n \abs{ \frac{\beta^1_n}{\abs{\beta^1_n}} - \frac{\beta^2_n}{\abs{\beta^2_n}}}=\infty.
	\end{equation}
	\item Spacetime translation property: it holds that $\lambda^{(1)}_n=\lambda^{(2)}_n, \beta^1_n=\beta^2_n$ and 
	\begin{equation}\label{eq:Transl_Div}
		\lim_{n\to\infty} \abs{L^{\beta^1_n}\big(\lambda^{(1)}_n(t^1_n-t^2_n), \lambda^{(1)}_n(x^1_n-x^2_n)\big)}=\infty.
	\end{equation}
\end{enumerate}
\end{defn}
Definition~\ref{def:orthogonality} is motivated by the following property.
\begin{prop}\label{prop:OrthoLfour} If $w_1, w_2\in \Lfour$ and $\Lambda^1_n, \Lambda^2_n$ are orthogonal sequences of transformations, then for all $\alpha, \beta\in[0, \infty)$ such that $\alpha+\beta=4$,
\begin{equation}\label{eq:ortho_aftermath}
\lim_{n\to \infty} \iint_{\R^{1+3}} \abs{\lambda^{(1)}_n w_1(\Lambda^1_n(t, x))}^\alpha \abs{\lambda^{(1)}_n w_2(\Lambda^2_n(t, x))}^\beta\, dtdx  = 0. 
\end{equation}
\end{prop}
We can now recast, using our notation, the aforementioned linear profile decomposition of Ramos.
\begin{thm}\label{thm:Linear_Profiles}
 Let $\fbold_n$ be a bounded sequence in $\PhaseSpace(\R^3)$. Then there exists an at most countable set 
 \begin{equation}\label{eq:LinProfiles}
 	\Set{ (\Fbold^j, (\Lambda^j_n)_{n\in\N}) : j=1, 2, 3, \ldots} ,
\end{equation}
where $\Fbold^j\in\PhaseSpace(\R^3)$ and the sequences $(\Lambda^j_n)$ are pairwise orthogonal in the sense of Definition~\ref{def:orthogonality}, such that, up to passing to a subsequence, 
 \begin{equation}\label{eq:profile_decomposition}
 	S\fbold_n=\sum_{j=1}^J \lambda^{(j)}_n (S\Fbold^j)\circ \Lambda^j_n + S\rbold^J_n,
\end{equation}
where the remainder term $\rbold^J_n$ satisfies the vanishing property
\begin{equation}\label{eq:profile_decomposition_smallness_remainder_term}
	\lim_{J\to\infty} \limsup_{n\to \infty} \norm{S\rbold^J_n}_{\Lfour} =0.
\end{equation}
Moreover, for each $J\ge 1$, we have  the  Pythagorean expansion, as $n\to \infty$,
\begin{equation}\label{eq:Pythagorean_energy}
\Ds\norm{\fbold_n}_\PhaseSpace^2 = \sum_{j=1}^J \norm{ \fbold^j}_\PhaseSpace^2 + \norm{\rbold^J_n}_\PhaseSpace^2 + o(1).
\end{equation}
\end{thm}
To use Theorem~\ref{thm:Linear_Profiles} with nonlinear solutions, we will need the following lemma. We recall from Proposition~\ref{prop:small_data_theory} that a \emph{solution} to~\eqref{eq:cubic_NLW} is a function $u\in L^4(\R^{1+3})$, with $\ubold\in C(\R; \PhaseSpace)$, that satisfies the fixed point equation 
\begin{equation}\label{eq:AppOneFixpt}
	u=S\fbold + \sigma \antibox(u^3),
\end{equation}
for a $\fbold \in \PhaseSpace(\R^3)$. We write $u=\Phi(\fbold)$. In particular, we are implicitly assuming that $u$ is a \emph{global} solution, in the sense that it is defined for all $t\in \R$. We will not consider non-global solutions.
\begin{lem}[Perturbation Lemma]\label{lem:perturb}
	Let $u=\Phi(\fbold)$. For $\Mtilde>0$, assume that 
	$\norm{\utilde}_\Lfour\le \Mtilde$, where $\utilde$ satisfies
	\begin{equation}\label{eq:perturb_sol}
	  \begin{array}{cc}
		\Ds \lim_{t\to -\infty} \lVert \ubold(t)-\utildebold(t)\rVert_\PhaseSpace=0, &\text{and }\norm{e}_{L^{4/3}(\R^{1+3})}\le \eps,
	  \end{array}
	\end{equation}
	where $e:=\Box\,\utilde -\sigma \utilde^3$ in distributional sense. Then
	\begin{equation}\label{eq:PertConcl}
		\norm{u-\utilde}_\Lfour + \sup_{t\in\R}\norm{\ubold(t)-\utildebold(t)}_{\Hcaldot^{1/2}}\le C(\Mtilde)\, \eps.
	\end{equation}	
\end{lem}
\begin{proof} The assumptions~\eqref{eq:perturb_sol} imply that $\utilde$ satisfies the fixed-point equation 
\begin{equation}\label{eq:perturb_fixpoint}
  \utilde= S\fbold + \sigma \antibox(\utilde^3) +\antibox e,
\end{equation}
so the difference $w:=\utilde-u$ satisfies $w=\sigma \antibox(\utilde^3-u^3)+\antibox e$.  We now estimate $w$ on a time interval $(-\infty, T)\subset \R$ via the Strichartz inequality~\eqref{eq:ConfStrichartz}, which holds on such time intervals because of Remark~\ref{rem:Inhomogeneous_Interval};
\begin{equation}\label{eq:w_est}
	\begin{split}
	\norm{w}_{L^4((-\infty, T)\times \R^3)}&\le C\eps + C\abs{\sigma} \lVert (\utilde +w)^3-\utilde^3\rVert_{L^\frac43((-\infty, T)\times \R^3)} \\ 
	&\le C(\eps + \norm{w}^3_{L^4((-\infty, T)\times \R^3)})+C \norm{\utilde^2 w}_{L^\frac43((-\infty, T)\times \R^3)}.
	\end{split}
\end{equation}
The Gronwall-type inequality of~\cite[Lemma~8.1]{FaXiCa11} now implies that
\begin{equation}\label{eq:CazenaveFangXi_aftermath}
	\norm{w}_{L^4((-\infty, T)\times \R^3)}\le C_{\Mtilde} (\eps + \norm{w}^3_{L^4((-\infty, T)\times \R^3)}).
\end{equation}
Therefore, if $T\in \R$ is such that $\norm{w}_{L^4((-\infty, T)\times \R^3)}\le 2C_{\Mtilde} \eps$, then $$\norm{w}_{L^4((-\infty, T)\times \R)}\le C_{\Mtilde}\eps + C_{\Mtilde}(2C_{\Mtilde}\eps)^3\le \frac32C_{\Mtilde}\eps, $$ 
provided that $\eps$ is sufficiently small. By the bootstrap method, this proves the inequality $\norm{w}_{L^4(\R^{1+3})}\le \tfrac32C_{\Mtilde}\eps$. 

The same argument with $\sup_{t\in\R}\norm{\wbold(t)}_{\Hcaldot^{1/2}}$ in place of $\norm{w}_{\Lfour}$ concludes the proof.
\end{proof}                                                                                                                                                                                                                                                                                                                

\begin{cor}\label{cor:nonlinear_profdecomp}
  Let $A>0$ be such that, if $\norm{\fbold}_{\Hcaldot^{1/2}} \le A$, then there exists a unique solution $u=\Phi(\fbold)$. If  $u_n=\Phi(\fbold_n)$ satisfies $\norm{\fbold_n}_{\Hcaldot^{1/2}}\le A$, we associate to each profile $(\Fbold^j, \Lambda^j_n)$ in~\eqref{eq:LinProfiles} the \emph{nonlinear profile}
  \begin{equation}\label{eq:NonlinProfile}
    U^j:=\Phi(\Fbold^j).
  \end{equation}
  Then
  \begin{equation}\label{eq:nonlinear_profdecomp}
   u_n(t, x) = \sum_{j=1}^J \lambda^{(j)}_nU^j(\Lambda^j_n(t, x))+ S\rbold^J_n(t, x) + h^J_n(t, x),
  \end{equation}
  where $\rbold^J_n$ is the same as in~\eqref{eq:profile_decomposition}, while $h^J_n$ is a sequence that satisfies the vanishing condition
  \begin{equation}\label{eq:nonlinear_remainder}
    \lim_{J\to \infty} \limsup_{n\to \infty}\Tonde{ \norm{h^J_n}_\Lfour + \sup_{t\in \R} \norm{\hbold^J_n(t)}_{\Hcaldot^{1/2}} }= 0.
  \end{equation}
\end{cor}
\begin{proof}
	To apply Lemma~\ref{lem:perturb}, we fix $J\in\N$ and we denote 
	\begin{equation}\label{eq:utilde_n}
		\utilde_n^J(t, x)=\sum_{j=1}^J \lambda^{(j)}_n U^j_n(\Lambda^j_n(t,x)) + S\rbold^J_n.
	\end{equation}
	By orthogonality of the sequences $\Lambda^j_n$ (see Proposition~\ref{prop:OrthoLfour}), and by the vanishing property~\eqref{eq:profile_decomposition_smallness_remainder_term} of $S\rbold^J_n$, we can find a sequence $\eps^J_n\ge 0$ satisfying $\lim_{J} \limsup_{n} \eps^J_n=0$ and such that
	\begin{equation}\label{eq:control_utilde}
		\begin{split}
			\norm{\utilde_n^J}_\Lfour^4& =\sum_{j=1}^J \norm{U^j}_\Lfour^4 + \eps^J_n \\  
			&\le C(\sum_{j=1}^J \norm{\Fbold^j}_{\Hcaldot^{1/2}}^2)^2 + \eps^J_n\le C_AA^4,
		\end{split}
	\end{equation}
where we used the estimate~\eqref{eq:crude_estimates} and the Pythagorean expansion~\eqref{eq:Pythagorean_energy}. We remark that the estimate~\eqref{eq:control_utilde} is uniform in $J$. In order to  apply the perturbation Lemma~\ref{lem:perturb}, we notice that, by~\eqref{eq:profile_decomposition}, 
\begin{equation}\label{eq:smallness_infinity}
	\lim_{t\to -\infty} \norm{ \ubold_n(t) - \utildebold_n^J(t)}_{\Hcaldot^{1/2}} = 0,
\end{equation}
and, moreover,
\begin{equation}\label{eq:utilde_eqn}
	\begin{split}
	e^J_n:&= \Box\, \utilde^J_n -\sigma(\utilde^J_n)^3\\ 
	&=-\sigma \Quadre{
		\Tonde{ 
			\sum_{j=1}^J \lambda^{(j)}_n U^j\circ \Lambda^j_n + S\rbold^J_n
			}^3 -\sum_{j=1}^J \Tonde{ \lambda^{(j)}_n U^j\circ \Lambda^j_n}^3 
			},
	\end{split}
\end{equation}
so, again by orthogonality of $\{ \Lambda^j_n : j=1\ldots J\}$ and vanishing of $S\rbold^J_n$, 
\begin{equation}\label{eq:ErrSmallness}
	\lim_{J\to \infty} \limsup_{n\to \infty} \norm{e^J_n}_\Lfour=0.
\end{equation}
We thus obtain~\eqref{eq:nonlinear_remainder}, concluding the proof.
\end{proof}                                                                                                                                                                                                                                                                                                                
\begin{rem}\label{rem:LfourOrtho}
 Proposition~\ref{prop:OrthoLfour} also implies that 
 \begin{equation}\label{eq:LfourProfiles}
    \norm{u_n}_\Lfour^4 = \sum_{j=1}^J \norm{U^j}^4_\Lfour  +\norm{S\rbold^J_n}_\Lfour^4 + \eps^J_n,
 \end{equation}
  where 
  \begin{equation}\label{eq:EpsilonJn}
    \lim_{J\to\infty} \limsup_{n\to\infty} \eps^J_n=0.
  \end{equation}
\end{rem}

\section{The $\PhaseSpace$ norm is not Lorentz-invariant}
The lemma which we prove in this section immediately implies the existence of smooth solutions $u$ to~\eqref{eq:cubic_NLW} such that $\norm{\ubold(t)}_\PhaseSpace$ is not preserved by time translations and Lorentzian transformations. We recall from Section~\ref{sym} that, for all $\alpha\in(-1, 1)$,
\begin{equation}\label{eq:Lalpha_recall}
	\begin{array}{cc}
		L^\alpha(t, x)=(\gamma t- \gamma \alpha x_1, \gamma x_1-\gamma\alpha t, x_2, x_3), &\text{where }\gamma=(1-\alpha^2)^{-1/2}.
	\end{array}
\end{equation}
\begin{lem}\label{lem:time_derivative}
  Let $u$ be a smooth global solution to $\Box\, u=\sigma u^3$ on $\R^{1+3}$. Then 
  \begin{equation}\label{eq:time_derivative}
    \frac{\partial}{\partial t_0}\norm{\ubold(t_0)}_{{\Hcaldot^{1/2}}}^2 = 2 \sigma \int_{\R^3}(-\Delta)^{-{1/2}}(u_t(t_0, \cdot))u^3(t_0, x)\, dx,
  \end{equation}
  and, letting $u_\alpha:=u\circ L^\alpha$, 
  \begin{equation}\label{eq:Lorentz_derivative}
  	\left.\frac{\partial}{\partial \alpha}\norm{\ubold_\alpha(t_0)}_\PhaseSpace^2\right|_{\alpha=0} = -2\sigma\int_{\R^3} x_1(-\Delta)^{-{1/2}}(u_t(t_0, \cdot))u^3(t_0, x)\, dx.
\end{equation}
\end{lem}

\begin{proof}
 We recall that $\ubold(t_0)$ denotes the pair $(u(t_0, \cdot), u_t(t_0, \cdot))$. Using the equation, we obtain 
 \begin{equation}\label{eq:deriv_wbold}
 	\partial_{t_0}\ubold(t_0)=(u_t(t_0, \cdot), \Delta u(t_0, \cdot) +\sigma u^3(t_0, \cdot)).
\end{equation}
 Therefore
 \begin{equation}\label{eq:DerivNorm}
 	\begin{split}
 	&\partial_{t_0}\norm{\ubold(t_0)}_{\Hcaldot^{1/2}}^2 = 2\Braket{ \ubold(t_0)| \partial_{t_0} \ubold(t_0)}_{{\Hcaldot^{1/2}}} \\
	&= 2\int_{\R^3}(-\Delta)^{1/2} u(t_0, x) u_t(t_0, x) \,dx +2\int_{\R^3}(-\Delta)^{-{1/2}}u_t(t_0, x)\Delta u(t_0, x)\,dx\\ &\quad+2\sigma \int_{\R^3} (-\Delta)^{-{1/2}}u_t(t_0, x) u^3(t_0, x)\, dx.
	\end{split}
\end{equation}
Since $(-\Delta)^{-{1/2}}\Delta=-(-\Delta)^{{1/2}}$, the first two summands cancel, yielding~\eqref{eq:time_derivative}. 

To prove~\eqref{eq:Lorentz_derivative}, we begin by observing that 
\begin{equation}\label{eq:Lorentz_derivative_ubold}
	\left.\partial_\alpha \ubold_\alpha(t_0)\right|_{\alpha=0} = -(x_1 \partial_{t_0} +t_0\partial_{x_1})\ubold(t_0)-(0, \partial_{x_1} u(t_0)).
\end{equation}
Integration by parts immediately shows that $\Braket{\ubold(t_0) | t_0\partial_{x_1} \ubold(t_0)}_\PhaseSpace=0$. So, reasoning as before and using~\eqref{eq:deriv_wbold},  we obtain
\begin{equation}\label{eq:Lorentz_computation}
	\begin{split}
		 &-\frac{1}{2}\partial_{\alpha=0} \norm{\ubold_\alpha(t_0)}_\PhaseSpace^2= \Braket{ \ubold(t_0) | x_1\partial_{t_0} \ubold(t_0) +(0, \partial_{x_1}u(t_0)}_\PhaseSpace \\ 
		&= \int_{\R^3} 
		u_t (-\Delta)^{-\frac12} (x_1 \Delta u)+ (-\Delta)^{\frac12} u x_1 u_t   +(-\Delta)^{-\frac12}u_t \partial_{x_1} u +\sigma (-\Delta)^{-\frac12} u_t x_1  u^3.
	\end{split}
\end{equation}
Now, using the elementary commutator identity $[(-\Delta)^{-\frac12}, x_1]=(-\Delta)^{-\frac32}\partial_{x_1}$, we see that the first three summands cancel. This completes the proof.
\end{proof}                                                                                                                                                                                                                                                                                                                
It is very easy to construct smooth solutions to~\eqref{eq:cubic_NLW} such that the derivatives in~\eqref{eq:time_derivative} and~\eqref{eq:Lorentz_derivative} do not vanish. For example, if $f_0\ne 0$ is a smooth function with compact support and $f_1=f_0^3$, then if $\eps>0$ is sufficiently small there exists a unique smooth  solution $u$ to 
\begin{equation}\label{eq:NLW_initial_value_probl_eps}
	\begin{cases} 
		\Box\, u = \sigma u^3, &\text{on } \R^{1+3}, \\
		\ubold(0)=\eps\fbold,
	\end{cases}
\end{equation}
and by~\eqref{eq:time_derivative}, $\partial_{t_0=0} \norm{\ubold(t_0)}_\PhaseSpace^2=\norm{f_0^3}_{\Hdot^{-1/2}}^2\!\ne 0$. Taking $f_1=x_1 f_0^3$, we analogously obtain a solution such that $\partial_{\alpha=0} \norm{\ubold_\alpha (0)}^2_\PhaseSpace=\norm{x_1f_0^3}_{\Hdot^{-1/2}}^2\! \ne 0$. 

\section{Geometry of the set of maximizers} \label{app:geometry}
In this section we use the notation
\begin{equation}\label{eq:Gammaalpha_notation}
	\Gammaalpha \fbold(x):=\lambda \Ph_\theta \left.\vbold(\lambda L^\beta(t-t_0, x-x_0))\right|_{t=0},
\end{equation}
where $v= S\fbold$, $\vbold=(v, \partial_t v)$, the phase shift operator $\Ph_\theta$ is defined in~\eqref{eq:PhaseSymmetry}, and $\alphabold=(\lambda, \theta, \beta , t_0, x_0)$ belongs to
\begin{equation}\label{eq:alpha_bold}
	\Arom:= (0, \infty)\times \SSS^1\times\{\beta\in\R^3\,:\,\abs\beta<1\}\times \R\times \R^3.
\end{equation}
As mentioned in the fifth section, the set $\Mrom$ of extremizers of the Strichartz inequality is  
\begin{equation}\label{eq:Mrom_Unitary}
	\Mrom=\Set{ c\Gammaalpha \fbold_0 : c\ge 0,\ \alphabold \in \Arom},
\end{equation}
where $\fbold_0=\abs{\SSS^3}^{-1/2}(2(1+\abs\cdot^2)^{-1}, 0)$; here the normalization factor ensures that $\fbold_0$ has unit norm. We remark that each $\Gammaalpha$ is a unitary operator of~$\PhaseSpace$ onto itself and that $\Gamma_{\!\obold}$ is the identity.
\begin{lem}\label{lem:alphabold_bijective}
	The map
	\begin{equation}\label{eq:parameterization}
		(c, \alphabold)\in (0, \infty)\times \Arom \mapsto c\Gammaalpha\fbold_0 \in \Mrom\setminus\{\obold\}
	\end{equation}
	is injective, hence a bijection.
\end{lem}
This lemma, which we will prove at the end of the section, implies that $\Mrom\setminus\{\obold\}$ is a smooth 10-dimensional manifold parameterized by~\eqref{eq:parameterization}. The tangent space at $\fboldstar\ne \obold$ is 
\begin{equation}\label{eq:tangent_space_at_cgammaalpha}
	\begin{array}{cc}
	\Ds T_{\fboldstar}\Mrom=\Span\Set{ \Gammaalpha\fbold_0, 	\partial_{\alpha_i} \Gammaalpha \fbold_0 : i=1, 2,\ldots, 9},& \fboldstar=c\Gammaalpha\fbold_0,\ c\ne 0.
	\end{array}
\end{equation}
We will require a further lemma, which follows immediately from the explicit computations of the third section of~\cite{Negro18}.
\begin{lem}\label{lem:non_singular_matrix}
	The matrix 
	\begin{equation}\label{eq:mat_scalprods}
		M_0:=\begin{bmatrix} \Braket{ \partial_{\alpha_i}\big\rvert_{\alphabold=\obold}\Gammaalpha \fbold_0 | \partial_{\alpha_j}\big\rvert_{\alphabold=\obold}\Gammaalpha \fbold_0}_{\PhaseSpace}
 \end{bmatrix}_{ i, j=1\ldots 9}
 	\end{equation}
	is nonsingular and positive definite.
\end{lem}

We can now state the main result of this section.
\begin{prop}\label{prop:Metric_Projection}
	For every $\fbold\in\PhaseSpace$ there exists $\fboldstar\in \Mrom$ such that 
	\begin{equation}\label{eq:distance_minimize}
		\norm{\fbold-\fboldstar}_\PhaseSpace =\dist(\fbold, \Mrom), 
	\end{equation}
	and, if $\fboldstar \ne \obold$, then $\fbold-\fboldstar\, \bot\, T_{\fboldstar}\Mrom$, that is
	\begin{equation}\label{eq:Minimizer_Orthogonal}
		\begin{array}{cc}
		\Braket{\fbold-\fboldstar | \gbold}_\PhaseSpace= 0, & \forall\ \gbold\in T_{\fboldstar} \Mrom.
		\end{array}
	\end{equation}
	Moreover, there is a constant $\rho\in(0,1)$ such that, if 
	\begin{equation}\label{eq:distance_small_norm}
		\dist(\fbold, \Mrom)< \rho\norm{\fbold}_\PhaseSpace, 
	\end{equation}
	then $\fboldstar$ is uniquely determined.
\end{prop}
\begin{proof}
The existence of $\fboldstar$ and the property~\eqref{eq:Minimizer_Orthogonal} have been proved in Step~1 in the proof of Proposition~5.3 of~\cite{Negro18}. To establish uniqueness, we assume that~\eqref{eq:distance_small_norm} holds for a constant $\rho$ to be determined, and we suppose that there exist $\fboldstar$ and $\fbold_{\!\star}'$ in $\Mrom\setminus\{\obold\}$ such that 
\begin{equation}\label{eq:norm_is_distance}
	\begin{array}{cc}
	\fbold=\fboldstar+\fboldbot=\fboldstarprime+\fboldbotprime, &\text{where }	\norm{\fboldbot}_\PhaseSpace=\norm{\fboldbotprime}_\PhaseSpace=\dist(\fbold, \Mrom).
	\end{array}
\end{equation}
Our goal is to show that $\fboldstar=\fboldstarprime$. We consider $\alphabold, \alphabold'\in\Arom$ such that 
\begin{equation}\label{eq:fboldstar_fboldstarprime}
	\begin{array}{cccc}
	\fboldstar=c\Gammaalpha \fbold_0& \text{and} & \fbold_{\!\star}'= c'\Gamma_{\!\alphabold'}\fbold_0, &\text{where } c=\norm{\fboldstar}_\PhaseSpace, c'=\norm{\fboldstarprime}_\PhaseSpace,
	\end{array}
\end{equation}
and, replacing $\fbold$ with $\Gamma_{\!\alphabold}^{-1}\fbold$ if needed, we can assume that $\Gammaalpha=\Gamma_{\!\obold}$. The orthogonality~\eqref{eq:Minimizer_Orthogonal} implies that $\Braket{\fboldbot|\fboldstar}_\PhaseSpace=\Braket{\fboldbotprime|\fboldstarprime}_\PhaseSpace=0$, so using~\eqref{eq:norm_is_distance} we can expand $\norm{\fbold}_\PhaseSpace^2$, yielding
\begin{equation}
	c=c'=\norm{\fboldstar}_\PhaseSpace=\norm{\fboldstarprime}_\PhaseSpace = \sqrt{\norm{\fbold}_\PhaseSpace^2- \dist(\fbold, \Mrom)^2}.
\end{equation}
It follows from these considerations that we can rewrite~\eqref{eq:norm_is_distance} as
\begin{equation}\label{eq:two_projections}
	\frac{\fbold}{c}=\fbold_0 + \frac{\fboldbot}{c} = \Gamma_{\!\alphabold'} \fbold_0+\frac{\fboldbotprime}{c},
\end{equation}
from which we infer the estimate 
\begin{equation}\label{eq:bound_alphaprime}
	\norm{\fbold_0-\Gamma_{\!\alphabold'}\fbold_0}_\PhaseSpace \le \frac{2\dist(\fbold, \Mrom)}{\sqrt{\norm{\fbold}_\PhaseSpace^2-\dist(\fbold, \Mrom)^2}}\le \frac{2\rho}{\sqrt{1-\rho^2}},
\end{equation}
and analogously,
\begin{equation}\label{eq:bound_difference_fbold_fboldzero}
	\norm{\fbold/c- \fbold_0}_\PhaseSpace\le \frac{\rho}{\sqrt{1-\rho^2}}.
\end{equation}
To finish the proof, it will suffice to show that $\alphabold'=\obold$. 

As a first step, we claim that
\begin{equation}\label{eq:alphaprime_bound}
	\abs{\alphabold'}\le C\norm{\fbold_0-\Gamma_{\!\alphabold'}\fbold_0}_\PhaseSpace,
\end{equation}
for a $C>0$. To prove this, we begin by squaring the left-hand side of \eqref{eq:bound_alphaprime},
\begin{equation}\label{eq:norm_expansion}
	\norm{\fbold_0-\Gamma_{\!\alphabold'}\fbold_0}_\PhaseSpace^2=2-2\Braket{\fbold_0 |\Gamma_{\!\alphabold'}\fbold_0},_\PhaseSpace
\end{equation}
so that $$\Braket{\fbold_0 |\Gamma_{\!\alphabold'}\fbold_0}_\PhaseSpace \ge \tfrac{1-3\rho^2}{1-\rho^2}.$$ Assuming, as we may,   that $\rho<1/\sqrt3$, the right-hand side of this inequality is strictly positive. Now, $\Braket{\fbold_0 |\Gamma_{\!\sigmabold}\fbold_0}_\PhaseSpace\to 0$ as $\abs\sigmabold\to \infty$; see for example~\cite[Lemmas 3.2 and 4.1]{Ramos12}. Thus, there must be a $C(\rho)>0$ such that $\lvert\alphabold'\rvert\le C(\rho)$. 

We can then assume, for a contradiction, that 
\begin{equation}\label{eq:contradiction_sequence}
	\begin{array}{cc}
	\Ds \frac{\norm{\fbold_0-\Gamma_{\!\alphabold_n} \fbold_0}_\PhaseSpace^2}{\abs{\alphabold_n}^2} \to 0, &\text{for a sequence }\alphabold_n\in\Arom,\ \abs{\alphabold_n}\le C(\rho).
	\end{array}
\end{equation}
There exists $\alphabold_0\in\Arom$ such that $\alphabold_n\to \alphabold_0$ up to a subsequence. If $\lvert\alphabold_0\rvert\ne 0$, then~\eqref{eq:contradiction_sequence} would imply that $\norm{\fbold_0-\Gamma_{\!\alphabold_0}\fbold_0}_\PhaseSpace=0$, but this is ruled out by Lemma~\ref{lem:alphabold_bijective}. The only remaining possibility is that $\abs{\alphabold_n}\to 0$. We record now two identities that hold for all $\alphabold\in \Arom$;
\begin{equation}\label{eq:unitary_trick}
 	\Braket{\Gammaalpha \fbold_0|\partial_{\alpha_i}\Gammaalpha \fbold_0}_\PhaseSpace= \partial_{\alpha_i}\tfrac12 \norm{\Gamma_{\!\alphabold} \fbold_0}^2_\PhaseSpace=0,
\end{equation}
where we used that $\Gammaalpha$ is unitary and
\begin{equation}\label{eq:unitary_trick_two}
	-\Braket{\Gammaalpha \fbold_0 | \partial_{\alpha_i}\partial_{\alpha_j}\Gammaalpha \fbold_0}_\PhaseSpace = \Braket{ \partial_{\alpha_i}\Gammaalpha \fbold_0 | \partial_{\alpha_j}\Gammaalpha \fbold_0}_{\PhaseSpace},
\end{equation}
which is obtained from~\eqref{eq:unitary_trick} by differentiating. Using these, we compute
\begin{equation}\label{eq:Ratio_Expansion}
	\norm{\fbold_0-\Gamma_{\!\alphabold}\fbold_0}_{\PhaseSpace}^2=2\sum_{i,j=1}^9\alpha_i\alpha_j\left.\Braket{\partial_{\sigma_i}\Gamma_{\!\sigmabold}\fbold_0 | \partial_{\sigma_j}\Gamma_{\!\sigmabold}\fbold_0}\right|_{\sigmabold=\obold}  + O(\abs{\alphabold}^3).
\end{equation}
Since the coefficients of the quadratic term are those of the matrix $M_0$ defined in~\eqref{eq:mat_scalprods}, the fact that $\abs{\alphabold_n}\to 0$ implies
\begin{equation}\label{eq:limit_ratio}
	0=\lim_{n\to \infty} \frac{\norm{\fbold_0-\Gamma_{\!\alphabold_n}\fbold_0}_{\PhaseSpace}^2}{\abs{\alphabold_n}^2}\ge2\lambda_0 >0, 
\end{equation}
where $\lambda_0$ is the minimal eigenvalue of $M_0$, which is strictly positive because of Lemma~\ref{lem:non_singular_matrix}. We have reached the desired contradiction and proved~\eqref{eq:alphaprime_bound}.

 To conclude the proof that $\alphabold'=\obold$, we define $\Fcal\colon \Arom \times \PhaseSpace\to \R^{9}$ by
\begin{equation}\label{eq:FcalFunction}
	\Fcal(\alphabold, \gbold):=\begin{bmatrix} \Ds
	\Braket{  \Gammaalpha\fbold_0  -\gbold|\partial_{\alpha_i}\Gammaalpha\fbold_0}_\PhaseSpace
	\end{bmatrix}_{i=1\ldots9}.
\end{equation}
By~\eqref{eq:Minimizer_Orthogonal}, $\Gamma_{\!\alphabold'}\fbold_0-\fbold/c=\fboldbotprime/c$ is orthogonal to the tangent space at $c\Gamma_{\!\alphabold'}\fbold_0$, which contains all the derivatives $\partial_{\alpha_i}\Gammaalpha \fbold_0$ at $\alphabold'$, so $\Fcal(\alphabold', \fbold/c)=0$. In the same way we see that $\Fcal(\obold, \fbold/c)=0$. 

Now, obviously, $\Fcal(\obold, \fbold_0)=0$. Using the identities~\eqref{eq:unitary_trick} and~\eqref{eq:unitary_trick_two} as before, we find that the Jacobian matrix $D_\alphabold \Fcal= \begin{bmatrix} \partial_{\alpha_j}\Fcal_i\end{bmatrix}_{i, j=1\ldots 9}$ at $(\obold, \fbold_0)$ is 
\begin{equation}\label{eq:Jacobian_Origin}
	D_\alphabold\Fcal(\obold, \fbold_0)=M_0,
 \end{equation}
so that, in particular, it is nonsingular. We can thus rewrite the identity $\Fcal(\alphabold', \fbold/c)=0$ as a fixed point relation;
\begin{equation}\label{eq:fixpt_eqq}
	\begin{array}{cc}
	\alphabold'=P(\alphabold', \fbold/c), & \text{where }P(\alphabold, \gbold):=\alphabold-D_\alphabold\Fcal(\obold, \fbold_0)^{-1}\Fcal(\alphabold, \gbold),
	\end{array}
\end{equation}
and the function $P$ is such that $D_\alphabold P(\obold, \fbold_0)=0$. Thus, there exists an absolute constant $\eps>0$ such that 
\begin{equation}\label{eq:matrixnorm_is_small}
	\begin{array}{cc}
		\norm{D_\alphabold P(\alphabold, \gbold)}\le\frac12, & \text{if }\abs{\alphabold}<\eps\text{ and }\norm{\gbold-\fbold_0}_\PhaseSpace <\eps.
	\end{array}
\end{equation}
Here, as is usual, the matrix norm is $\norm{M}:=\sup\Set{ \abs{Mx}/\abs{x} : x\in\R^9}$. We now require, as we may, that $\rho$ satisfies the additional condition 
\begin{equation}\label{eq:rho_bound}
	\frac{\rho}{\sqrt{1-\rho^2}}\le \frac{\eps}{2C},
\end{equation}
so that, combining~\eqref{eq:bound_alphaprime} and~\eqref{eq:alphaprime_bound}, we see that $\lvert t \alphabold'\rvert <\eps$ for all $t\in[0, 1]$, and moreover, $\norm{\fbold/c-\fbold_0}_\PhaseSpace<\eps$ by~\eqref{eq:bound_difference_fbold_fboldzero}. Thus $\norm{D_\alphabold P(t\alphabold', \fbold/c)}\le\tfrac12$, and from
\begin{equation}\label{eq:towards_alphaboldprime_vanish}
	\alphabold'=P(\alphabold', \fbold/c)=\int_0^1\frac{d}{dt}P(t\alphabold', \fbold/c)\, dt=\int_0^1 D_\alphabold P(t\alphabold', \fbold/c)\alphabold'\, dt, 
\end{equation}
where we used that $P(\obold, \fbold/c)=\obold$, we infer that
\begin{equation}\label{eq:estimate_alphaboldprime}
	\abs{\alphabold'}\le \int_0^1 \norm{D_\alphabold P(t\alphabold', \fbold/c)}\abs{\alphabold'}\, dt \le \frac12\abs{\alphabold'},
\end{equation}
so that $\lvert \alphabold' \rvert=0$, completing the proof.
\end{proof}
We now give a proof of Lemma~\ref{lem:alphabold_bijective}. We need to show that $c\Gammaalpha\fbold_0=c'\Gamma_{\!\alphabold'}\fbold_0$ implies that $c=c'$ and $\alphabold=\alphabold'$. Now, the first identity is immediate, as $c=\norm{c\vbold(0,\cdot)}_\PhaseSpace=\norm{c'\vbold'(0, \cdot)}_\PhaseSpace=c'$. Reasoning like in the proof of Proposition~\ref{prop:Metric_Projection}, we can also assume that $\Gamma_{\!\alphabold'}=\Gamma_\obold$. We are thus reduced to prove that
\begin{equation}\label{eq:injectivity}
	v_\alphabold= v_\obold\quad \Rightarrow\quad \alphabold=\obold,
\end{equation} 
where
\begin{equation}\label{v_explicit}
	\begin{array}{cc}
		v_\alphabold(t, x):=\lambda v_\theta(\lambda L^\beta(t-t_0, x-x_0)),& \alphabold=(\lambda, \theta, \beta, t_0, x_0)
	\end{array}
\end{equation}
and $v_\theta=S(\Ph_\theta \fbold_0)$. 
	We recall the energy-momentum relation 
\begin{equation}\label{eq:energy_momentum}
	( E(v_\alphabold), \Pbold(v_\alphabold) )=\lambda L^{-\beta}(E(v_\obold), \Pbold(v_\obold)),
\end{equation}
where
\begin{equation}\label{eq:energy_and_momentum_definition}
	\begin{array}{cc}
		\Ds E(v):=\int_{\R^3} \Tonde{\abs{\nabla v}^2+ (\partial_t v)^2}dx, &\Ds \Pbold(v):=\int_{\R^3}\partial_t v\,\nabla v\, dx;
	\end{array}
\end{equation}
see, for example,~\cite[Remark~2.5]{KiStVi12}. Since $v_\obold$ is radial, $\Pbold(v_\obold)=0$. Now, since $v_\alphabold=v_\obold$, then obviously $(E(v_\alphabold), \Pbold(v_\alphabold))=(E(v_\obold), \Pbold(v_\obold))$, so~\eqref{eq:energy_momentum} gives
\begin{equation}\label{eq:linear_eqs}
	\begin{array}{cccc}
		\lambda \gamma E(v_\obold) = E(v_\obold), & 
		 \lambda \gamma \beta E(v_\obold)=0, 
	&
	\text{where }\gamma:=(1-\abs\beta^2)^{-1/2},
	\end{array}
\end{equation}
from which we infer that $\lambda=1$ and $\beta=0$. To conclude, we equate the spatial Fourier transforms of $v_\theta(t-t_0, \cdot-x_0)$ and $v_0(t, \cdot)$;
\begin{equation}\label{eq:Fourier_vtheta}
	\begin{array}{cc}
	\cos((t-t_0)\abs\xi +\theta)e^{-ix_0\cdot \xi}\hat{f}_0(\xi)=\cos(t\abs\xi)\hat{f}_0(\xi), &\forall \xi\in \R^3,\,t\in\R,
	\end{array}
\end{equation}
where $f_0:=C(1+\abs{\cdot}^2)^{-1}$, so $\hat{f}_0(\xi)=Ce^{-\abs{\xi}}/\abs\xi$ for an irrelevant $C>0$, and in particular, $\hat{f}_0(\xi)\ne 0$ almost everywhere. This is only possible if $t_0=0, x_0=0$ and $\theta=0$ modulo $2\pi$, completing the proof.

\printbibliography
\end{document}